\documentclass{amsart}
\usepackage{amssymb,amsmath,mathrsfs,amsfonts}
\usepackage{amscd, amssymb, amsmath, amsthm, graphics}
\usepackage{amsmath,amsfonts,amsthm,amssymb}
\usepackage[all]{xy}

\usepackage[colorlinks=true]{hyperref}

\newtheorem{theorem}{Theorem}[section]
\newtheorem{lemma}[theorem]{Lemma}
\newtheorem{proposition}[theorem]{Proposition}
\newtheorem{corollary}[theorem]{Corollary}

\newtheorem*{theorem*}{Theorem}

\theoremstyle{definition}
\newtheorem{definition}[theorem]{Definition}

\newtheorem{remark}[theorem]{Remark}
\newtheorem{problem}[theorem]{Problem}

\renewcommand{\t}{ \widetilde}
\renewcommand{\hat}{ \widehat}

\newcommand{\Z}{{\mathbb Z}}
\newcommand{\R}{{\mathbb R}}

\newcommand{\C}{{\mathbb C}}

\newcommand{\Hi}{{\bf H}}

\newcommand{\Sft}{{\widetilde{\mathrm{SL}_2(\mathbb{R})}}}

\renewcommand{\S}{\bf S}

\newcommand{\co}{\colon\thinspace}

\newcommand{\SV}{{\mathrm{SV}}}
\newcommand{\CSV}{{\mathrm{CSV}}}
\newcommand{\HV}{{\mathrm{HV}}}
\newcommand{\CHV}{{\mathrm{CHV}}}

\begin{document}

\title[Simplicial volume and virtual Seifert volume for $3$-manifolds]{Positive simplicial volume implies virtually positive Seifert volume for $3$-manifolds}


\author{Pierre Derbez}
\address{LATP UMR 7353, 39 rue  Joliot-Curie, 13453 Marseille Cedex 13, France}
\email{pderbez@gmail.com}
\author{Yi Liu}
\address{Beijing International Center for Mathematical Research, Peking University, Beijing 100871, China}
\email{liuyi@math.pku.edu.cn}
\author{Hongbin Sun}
\address{Department of Mathematics, University of California at Berkeley, CA 94720, USA}
\email{hongbins@math.berkeley.edu}
\author{Shicheng Wang}
\address{Department of Mathematics, Peking University, Beijing 100871, China}
\email{wangsc@math.pku.edu.cn}





\begin{abstract}
%
%

In this paper, it is shown that for any closed orientable $3$-manifold
with positive simplicial volume, the growth of the Seifert volume of its finite
covers is faster than the linear rate. In particular, each closed orientable $3$-
manifold with positive simplicial volume has virtually positive Seifert volume.
The result reveals certain fundamental difference between the representation
volume of the hyperbolic type and the Seifert type. The proof is based on
developments and reactions of recent results on virtual domination and on
virtual representation volumes of $3$-manifolds.
\end{abstract}

\maketitle
\tableofcontents

\section{Introduction}

The representation volume of 3-manifolds is a beautiful
theory, exhibiting rich connections with many branches of
mathematics. The behavior of those volume functions
appears to be quite mysterious;
for example, their values are hard to predict
except in very few nice cases. On the other hand,
for most motivating applications,
it suffices to estimate the growth of such volumes for finite covers
of the considered 3-manifold.
In this paper, we are intended to investigate the possibility
of the latter, which is  as interesting
as a topic on its own right.


To be more specific,
let us introduce some basic notations 
and mention some known properties of the representation volume.
Let $G$ be either
$${\rm Iso}_+{\Hi}^3\cong{\rm PSL}(2;{\C}),$$
the orientation-preserving isometry group of
the 3-dimensional hyperbolic geometry, or
$${\rm Iso}_e\t{\rm SL_2(\R)}\cong\R\times_\Z\t{{\rm SL}_2(\R)},$$
the identity component of the isometry group of the Seifert geometry. 
For any closed orientable 3-manifold $N$ and any representation $\rho\co\pi_1N\to G$,
denote by ${\rm vol}_G(N,\rho)$ the (unsigned) volume of $\rho$.
We denote the set of $G$--representation volumes of $N$ by
$${\rm vol}(N,G)\,=\,\left\{{\rm vol}_G(N,\rho)\,:\,\rho \textrm{ any representation}\ \pi_1N\to G\right\},$$
which is a subset of the interval $[0,+\infty)$.

The following theorem contains a collection of fundamental facts in the theory of representation volumes
\cite{BG1,Re-rationality}. 

\begin{theorem}\label{basic property of volume of presentation}
	Let $N$ be a closed orientable $3$-manifold.
	\begin{enumerate}
		\item The sets of values ${\rm vol}(N,{\rm Iso}_+{\Hi}^3)$ and ${\rm vol}(N, {\rm
		Iso}_e\t{{\rm SL}_2(\R)})$ are both finite.
		Hence the following  values
		$${\HV}(N)\,=\,\max{\rm vol}\left(N,{\rm Iso}_+{\Hi}^3\right)$$
		and
		$${\SV}(N)\,=\,\max{\rm vol}\left(N, {\rm Iso}_e\t{{\rm SL}_2(\R)}\right)$$
		exist in $[0,+\infty)$ depending only on $N$.
		\item If $N$ admits a hyperbolic geometric structure, then ${\HV}(N)$ equals
		the usual hyperbolic volume of $N$,
		reached by any discrete and faithful representation.
		A similar statement holds for $\SV(N)$ when $N$ admits a Seifert geometric structure.
		\item If $P_1,\cdots,P_s$ are the prime factors of $N$ in the Kneser--Milnor decomposition,
		then
		$${\HV}(N)\,=\,{\HV}(P_1)+\cdots+{\HV}(P_s).$$
		A similar formula holds for $\SV(N)$.
		\item
		For any map $f\co M\to N$ between closed orientable $3$-manifolds,
		$${\HV}(M)\geq|{\rm deg}f|\cdot{\HV}(N).$$
		The same comparison holds for $\SV(M)$ and $\SV(N)$.
	\end{enumerate}
\end{theorem}

The values ${\HV}(N)$ and ${\SV}(N)$
in the conclusion of
Theorem \ref{basic property of volume of presentation} (1)
are called the \emph{hyperbolic volume} and the \emph{Seifert volume} of $N$, respectively.
In light of Theorem \ref{basic property of volume of presentation} (3),
we assume from now on that all the closed orientable 3-manifolds
considered are prime, unless specified otherwise.
This is especially convenient when we speak of the geometric decomposition of the 3-manifold.

%

\begin{remark}\label{relation with $D(M,N)$}
	Representation volumes are introduced and studied by R.~Brooks and W.~Goldman \cite{BG1,BG2}
	as a generalization of the simplicial volume originally due to M.~Gromov \cite{Gr}.
	Among the eight 3-dimensional geometries of W.~P.~Thurston, ${\Hi}^3$ and $\t{{\rm SL}_2(\R)}$
	are the only two that yield nontrivial invariants, the hyperbolic volume and the Seifert volume,
	accordingly.
	Recall that the \emph{simplicial volume} of a closed orientable 3-manifold $N$ is defined to be
	$v_3\|N\|$, the product of the Gromov norm $\|N\|$ of $N$
	and the volume $v_3$ of the ideal regular hyperbolic tetrahedron;
	it is known to equal to	the sum of the classical hyperbolic volume of the hyperbolic pieces
	\cite{So}.
	
	Like the simplicial volume, the volumes of Brooks--Goldman satisfy the \emph{domination property},
	as stated by Theorem \ref{basic property of volume of presentation} (4).
	It follows that if either of the volumes ${\HV}(N)$ or ${\SV}(N)$
	is positive, then	the set of mapping degrees $D(M,N)$
	of $N$ by any given 3-manifold $M$ must be finite.
	Unlike the simplicial volume,
	neither the hyperbolic volume nor the Seifert volume satisfies the \emph{covering property},
	\cite[Corollary 1.8]{DLW};
	see Section \ref{Sec-onCoveringInvariants} for some further discussion.
\end{remark}

It can be inferred from Theorem \ref{basic property of volume of
presentation} and Remark \ref{relation with $D(M,N)$}
that non-vanishing $\HV(N)$ or $\SV(N)$
contains interesting information about the topology of the 3-manifold $N$.
However, such information seems difficult to characterize.
For example, the vanishing or non-vanishing of $\SV(N)$
implies nothing about the behavior of ${\HV}(N)$, \cite[Section 4 and 5]{BG1};
and except for the geometric case
(Theorem \ref{basic property of volume of presentation} (2)),
the geometry of pieces fails
to detect the vanishing or non-vanishing of $\HV(N)$ or $\SV(N)$ either,
\cite[Theorem 1.7]{DLW}.
On the other hand, the existence
of some finite cover of $N$ with non-vanishing representation volume
turns out to be a question more accessible.
An affirmative answer would be practically useful:
It implies the finiteness of the set of mapping degrees as before.
Motivated by that application, 
it has been discovered that any non-geometric graph manifold admits a finite cover of positive Seifert volume,
\cite{DW1,DW2}; a much more general construction
that invokes Chern--Simons theoretic calculations and virtual properties of 3--manifolds
shows that a right geometric piece implies
virtually positive volume of the right geometry \cite[Theorems 1.6]{DLW}:
%

\begin{theorem}\label{non-zero and zero} 
Suppose that $N$  is a closed orientable non-geometric prime 3-manifold.
\begin{enumerate}
\item
	If $N$ contains at least one hyperbolic geometric piece,
        then the hyperbolic volume of some finite cover of $N$ is positive.
\item	
	If $N$ contains at least one Seifert geometric piece,
        then the Seifert volume of some finite cover of $N$ positive.
\end{enumerate}
\end{theorem}

Despite the seeming parallelism so far, the hyperbolic volume and
the Seifert volume behave drastically differently with respect to finite covers.
In this paper,
we support this point by investigating two problems proposed in \cite[Section 8]{DLW}:

\begin{problem}\label{control}
	Estimate the growth of virtual hyperbolic volume and virtual Seifert volume.
\end{problem}

\begin{problem}\label{virt-pos-seif}
	Is the Seifert volume of a closed prime $3$-manifold
	virtually positive if it has positive simplicial volume?
\end{problem}

The main results of this paper address Problem \ref{virt-pos-seif} affirmatively
(Theorem \ref{virt-positive})
and Problem \ref{control} partially for 3-manifolds of positive simplicial volume
(Theorem \ref{main-unbounded} and Remark \ref{remark-unbounded}),
showing that the growth of virtual Seifert volume is super-linear while the growth
of virtual hyperbolic volume is linear.
On Problem \ref{virt-pos-seif}, the case of closed hyperbolic 3-manifolds
is already known as a direct consequence of the much stronger virtual domination theorem
\cite{Sun} (quoted as Theorem \ref{virt-dom} below);
so essentially it remains to treat the case of non-geometric 3-manifolds
(with only hyperbolic pieces).
On Problem \ref{control}, it is easy to
observe that the growth of virtual Seifert volume for a closed Seifert geometric 3-manifold
is linear, indeed, in a constant rate equal to its Seifert volume.
Comparing with our result,
we are left with the impression that the growth of virtual hyperbolic volume
might be largely governed by the simplicial volume, and the growth of virtual Seifert volume
appears to be more sensitive to the geometric decomposition.

The main results of this paper are stated as the following Theorems \ref{virt-positive} and \ref{main-unbounded}:

\begin{theorem}\label{virt-positive}
  If $M$ is a closed orientable 3-manifold with positive simplicial volume, then there is a finite cover $\tilde{M}$ of $M$
  with positive Seifert volume.
\end{theorem}

Combining with results of \cite{DW2,DLW,DSW}, we infer immediately the following characterization:

\begin{corollary}
 Suppose that $N$ is a closed orientable 3-manifold. Then the following three statements are equivalent.
\begin{enumerate}
	\item The set of mapping degrees $D(M,N)$ is finite for every closed orientable 3-manifold $M$.
	\item The Seifert volume of some finite cover of $N$ is positive.
	\item At least one prime factor of $N$ is Seifert geometric, or hyperbolic, or non-geometric.
\end{enumerate}
\end{corollary}

\begin{theorem}\label{main-unbounded}
For any closed oriented $3$-manifold $M$ with non-vanishing simplicial volume,
the set of values
$$\left\{\frac{\SV(M')}{[M':M]}\,:\, M'\textrm{ any finite cover of }M\right\}$$
has no upper bound in $[0,+\infty)$.
\end{theorem}

\begin{remark}\label{remark-unbounded}
In contrast, it is evident by \cite[Theorem B]{Re-rationality} and Theorem \ref{basic property of volume of presentation}
that the set of values
$$\left\{\frac{\HV(M')}{[M':M]}\,:\, M'\textrm{ any finite cover of }M\right\}$$
is contained by the interval $[\,0,\,v_3\|M\|\,]$ where the upper bound is the simplicial volume of $M$.
\end{remark}

Theorem \ref{main-unbounded} is significantly stronger than Theorem \ref{virt-positive}.
Let us take a closer look at the geometric case to illustrate
their difference in the proof.
As mentioned, when $M$ is assumed to be geometric, hence hyperbolic,
Theorem \ref{virt-positive} is implied by the following result
of \cite{Sun}, by taking $N$ to be a target with positive Seifert volume:

\begin{theorem}\label{virt-dom}
	For any closed oriented hyperbolic 3-manifold $M$, and any closed oriented 3-manifold $N$,
  $M$ admits a finite cover $\tilde{M}$,
	such that there exists a $\pi_1$-surjective degree-$2$ map $f:\tilde{M}\rightarrow N$.
\end{theorem}

Even though Theorem \ref{virt-dom} is a powerful construction,
employing deep theories including \cite{KM,LM} and \cite{Ag,Wi}
on building and separating certain quasiconvex subgroups
in closed hyperbolic 3-manifold groups,
the construction provides no control on the degree $[\tilde{M}:M]$.
So Theorem \ref{main-unbounded} stays beyond the reach of Theorem \ref{virt-dom}.
Armed with a more recent result of A.~Gaifullin \cite{Ga},
we prove the following Theorem \ref{efficient-virt-domi} based on Theorem \ref{virt-dom}.
The improved construction is supplied with a desired efficient control of the mapping degree:

\begin{theorem}\label{efficient-virt-domi}
	For any closed oriented hyperbolic $3$--manifold $M$, there exists a positive constant $c(M)$ such that the following statement holds. For any closed
	oriented $3$--manifold $N$ and any $\epsilon>0$, there exists a finite cover $M'$ of $M$ which admits a non--zero degree map $f:M'\rightarrow N$, such
	that $$\|M'\|\leq c(M)\cdot|\mathrm{deg}(f)|\cdot(\|N\|+\epsilon).$$
\end{theorem}
%

To prove Theorems \ref{virt-positive} and \ref{main-unbounded} in the non-geometric case,
it is tempting to extend Theorems \ref{virt-dom} and \ref{efficient-virt-domi}
to mixed 3-manifolds, but we do not have available tools for that project.
Instead, we follow the framework of \cite{DLW}
and integrate the virtual domination theorems.
The reaction between Theorem \ref{virt-dom}
and the fundamental construction for Theorem \ref{non-zero and zero}
is fairly direct and illustrating,
so we present it and prove Theorem \ref{virt-positive} as a warm-up.
The proof of Theorem \ref{main-unbounded} is relatively more sophisticated,
not only because of Theorem \ref{efficient-virt-domi},
but requiring some details of \cite{DLW}.
In particular, we introduce an auxiliary notion
called \emph{CI completions}
to formalize some useful idea underlying the construction of \cite{DLW},
(see Subsection \ref{Subsec-CIcompletion}).

All the arguments are based on explicitly stated results,
and the exposition is kept otherwise self-contained.
The organization of this paper is as the following.
The proofs of Theorem \ref{virt-positive}, and Theorem \ref{efficient-virt-domi}, and Theorem \ref{main-unbounded}
occupy Sections \ref{Sec-Theorem-virt-positive}, \ref{Sec-Theorem-efficient-virt-domi}, and \ref{Sec-Theorem-main-unbounded},
respectively.
Section \ref{Sec-preliminaries} includes preliminaries on 3-manifold topology and representation volume.
Section \ref{Sec-onCoveringInvariants} contains some further questions and observations.

\subsection*{Acknowledgement}
We are grateful to Ian Agol for valuable conversations.
The second author is supported by the Recruitment Program of Global Youth Experts
of China. The third author is partially supported by Grant No.~DMS-1510383 of the National
Science Foundation of the United States.
The last author is partially supported by Grant No.~11371034
of the National Natural Science Foundation of China.

\section{Preliminaries}\label{Sec-preliminaries}
	In this section, we review the geometric decomposition of $3$-manifolds and the theory of representation volumes.

	\subsection{Geometry and topology of 3-manifolds after Thurston}
		Let $N$ be a connected compact prime orientable 3-manifold with toral or empty boundary.
		As a consequence of the geometrization of $3$-manifolds \cite{Th1,Th2}
		achieved by G.~Perelman and W.~Thurston,
		exactly one of the following case holds:
		\begin{itemize}
			\item Either $N$ is geometric, supporting one of the following eight geometries: ${\Hi}^3$,
			$\widetilde{{\rm SL}_2({\R})}$, ${\Hi}^2\times{\R}$, ${\rm Sol}$, ${\rm Nil}$, ${\R}^3$, ${\S}^3$ and
			${\S}^2\times {\R}$ (where ${\Hi}^n$, ${\R}^n$ and ${\S}^n$ are the $n$-dimensional
			hyperbolic space, Euclidean space, and spherical space respectively);
			\item or $N$ has a canonical nontrivial geometric decomposition.
			In other words, there is a nonempty minimal union $\mathcal{T}_N\subset N$ of disjoint
			essential tori and Klein bottles in $N$, unique up to isotopy, such
			that each component of $N\setminus\mathcal{T}_N$ is either Seifert fibered or atoroidal.
			In the Seifert fibered case, the piece supports  the ${\Hi}^2\times{\R}$ geometry and
			the $\widetilde{{\rm SL}_2({\R})}$ geometry. In the atoroidal case, the piece supports the ${\Hi}^3$ geometry.
		\end{itemize}
		When $N$ has nontrivial geometric decomposition, we call the components of $N\setminus\mathcal{T}_N$
		the \emph{geometric pieces} of $N$, or more specifically,
		\emph{Seifert pieces} or \emph{hyperbolic pieces} according to their geometry.
		
        Traditionally, there is another decomposition introduced
		by Jaco--Shalen \cite{JS} and Johannson \cite{Joh}, known as the \emph{JSJ decomposition}.
		When $N$ contains no essential Klein bottles and has a nontrivial geometric decomposition,
		the JSJ decomposition of $N$ coincides with its geometric decomposition,
		so the cutting tori and the geometric pieces are also the \emph{JSJ tori} and
		the \emph{JSJ pieces}, respectively.
		Possibly after passing to a double cover of $N$, we may assume that $N$ contains
		no essential Klein bottle.
		
		A hyperbolic piece $J$ can be realized
		as a complete hyperbolic $3$-manifold of finite volume,
		unique up to isometry (by Mostow Rigidity).
		Let $J$ be a compact, orientable $3$-manifold whose boundary
		consists of tori $T_1,\ldots,T_p$ and whose interior admits a complete
		hyperbolic metric. Identify $J$ with the complement of $p$ disjoint cusps in the corresponding
        hyperbolic manifold, then $\partial J$ has a Euclidean metric induced
		from the hyperbolic structure, and each closed Euclidean geodesic in
		$\partial J$ has the induced length.
		The Hyperbolic Dehn Filling Theorem of Thurston \cite[Theorem 5.8.2]{Th1}
		can be stated in the following form.
		
		\begin{theorem}\label{surgered}
			There is a positive constant $C$ such that
 			the closed 3-manifold  $J(\zeta_1,\ldots,\zeta_n)$ obtained by
 			Dehn filling each $T_i$ along a slope $\zeta_i\subset T_i$ admits a complete hyperbolic
			structure if each $\zeta_i$ has length greater than $C$.
			Moreover, with suitably chosen base points, $J(\zeta_1,\ldots,\zeta_n)$
			converges to the corresponding cusped hyperbolic $3$-manifold in the Gromov--Hausdorff sense as the minimal length
			of $\zeta_i$ tends to infinity.
		\end{theorem}
		
 		A Seifert piece $J$ of a non-geometric prime closed $3$-manifold $N$ supports
 		both the ${\Hi}^2\times{\R}$ geometry and the $\widetilde{{\rm SL}_2({\R})}$ geometry.
 		In this paper, we are more interested in the latter case, so we describe the structure
 		of $\widetilde{{\rm SL}_2(\R)}$ geometric manifolds in the following.

We consider the group ${\rm PSL}(2;{\R})$ as the orientation
preserving isometries of the hyperbolic
  $2$-space ${\Hi}^2=\{z\in{\C}|\  \Im(z)>0\}$ with $i$ as a base point. In this way ${\rm PSL}(2;{\R})$ is a (topologically trivial) circle bundle over ${\Hi}^2$.
  Let $p\co\t{{\rm SL}_2(\R)}\to{\rm PSL}(2;{\R})$ be the universal covering of ${\rm PSL}(2;{\R})$ with the induced metric, then $\t{{\rm SL}_2(\R)}$
  is a  line bundle over ${\Hi}^2$. For any $\alpha\in{\R}$, denote by ${\rm sh}(\alpha)$ the element of $\t{{\rm SL}_2(\R)}$
  whose projection into ${\rm PSL}(2;{\R})$ is given by $\begin{pmatrix}
\cos(2\pi\alpha)&\sin(2\pi\alpha) \\
-\sin(2\pi\alpha)&\cos(2\pi\alpha)
\end{pmatrix}$. Then the set $\{{\rm sh}(n)|\  n\in{\Z}\}$
  is the kernel of $p$, as well as the center of $\t{{\rm SL}_2(\R)}$, acting by integral translation along the fibers of $\t{{\rm SL}_2(\R)}$.
  By extending this ${\Z}$-action on the fibers by the ${\R}$-action, we get the whole identity component of the isometry group of $\t{{\rm SL}_2(\R)}$.
  To summarize, we have the following diagram of central extensions
$$\xymatrix{
\{0\} \ar[r] \ar[d] & {\Z} \ar[r] \ar[d] & \t{{\rm SL}_2(\R)}\ar[r] \ar[d] & {\rm PSL}(2;{\R})\ar[r] \ar[d] & \{1\} \ar[d] \\
\{0\} \ar[r]  & {\R} \ar[r]  & {\rm Iso}_e\t{{\rm SL}_2(\R)}
\ar[r]  & {\rm PSL}(2;{\R})\ar[r]  & \{1\} }.$$ In particular, the
group ${\rm Iso}_e\t{{\rm SL}_2(\R)}$ is generated by $\t{{\rm
SL}_2(\R)}$ and the image of ${\R}$ which intersect with each other in the
image of ${\Z}$. More precisely, we state the following useful lemma which is easy to check.

\begin{lemma}\label{SL}
We have the identification ${\rm Iso}_e\t{{\rm SL}_2(\R)}={\R}\times_{\Z}\t{{\rm SL}_2(\R)}$:
where
$(x,h)\sim({x'},h')$ if and only if there exists an integer
$n\in{\Z}$ such that ${x'}-x=n$ and $h'={\rm sh}(-n)\circ h$.
\end{lemma}

From \cite{BG2} we know that a closed orientable $3$-manifold $J$ supports the $\widetilde {{\rm
SL}_2(\R)}$ geometry, i.e.~there is a discrete and faithful representation
$\psi: \pi_1J\to {\rm Iso} \t{{\rm SL}_2(\R)}$, if and only if
$J$ is a Seifert fibered space with non-zero Euler number $e(J)$ and the base orbifold $\chi_{O(J)}$ has
negative Euler characteristic.

\subsection{Representation volumes of closed manifolds}
In this subsection, we recall the definition of volume of representations.
There are a few equivalent definitions, and we will only state one of them.

Given a semi-simple, connected Lie  group $G$ and a closed oriented
manifold $M^n$ of the same dimension as  the contractible space
$X^n=G/K$, where $K$ is a maximal compact subgroup of $G$. We can
associate to each representation $\rho\co\pi_1M\to G$, a volume
${\rm vol}_G(M,\rho)$ in the following way.

First  fix a
$G$-invariant Riemannian metric $g_X$ on $X$, and denote by
$\omega_X$ the corresponding $G$-invariant volume form. Let $\t{M}$ denote the universal covering of $M$. We think of the elements
$\t{x}$ of  $\t{M}$
as the homotopy classes of paths $\gamma\co[0,1]\to M$ with $\gamma(0)=x_{0}$ which are acted by $\pi_1(M,x_0)$ by setting $[\sigma].\t{x}=[\sigma.\gamma]$,
where $.$ denotes the composition of paths.

A developing map $D_{\rho}\co\t{M}\to X$ associated to $\rho$ is a
$\pi_1M$-equivariant map such that
for any $x\in \t{M}$ and $\alpha \in \pi_1M$, we have
$$D_{\rho}(\alpha.x)=\rho(\alpha)D_{\rho}(x)$$ where $\rho(\alpha)$ acts on $X$ as an isometry.
Such a map does exist and can be constructed explicitly as in
\cite{BCG}: Fix a triangulation $\Delta_M$ of $M$, then it lifts to a
triangulation $\Delta_{\t{M}}$ of $\t{M}$, which is $\pi_1M$-invariant.
Then fix a fundamental domain $\Omega$ of $M$ in $\t{M}$ such that
the zero skeleton $\Delta^0_{\t{M}}$ misses the frontier of $\Omega$. Let
$\{x_1,\ldots,x_l\}$ be  the vertices of $\Delta^0_{\t{M}}$ in $\Omega$, and
let $\{y_1,\ldots,y_l\}$ be  any $l$ points in $X$.  We first set
$$D_{\rho}(x_i)=y_i, \,\, i=1 ,\ldots, l.$$
Then extend $D_{\rho}$ in a $\pi_1M$-equivariant way to $\Delta^0_{\t{M}}$:
for any vertex $x$ in $\Delta^0_{\t{M}}$, there is a unique  vertex
$x_i$ in $\Omega$ and $\alpha_x\in \pi_1M$ such that
$\alpha_x.x_i=x$, and we set
$D_{\rho}(x)=\rho(\alpha_x)^{-1}D_{\rho}(x_i)$. Finally we extend
$D_{\rho}$ to edges, faces, etc., and $n$-simplices of $\Delta_{\t{M}}$
by straightening their images to totally geodesics objects using the homogeneous
metric on the contractible space $X$.
This  map is unique up to equivariant homotopy. Then
$D_{\rho}^{\ast}(\omega_X)$ is a $\pi_1M$-invariant closed $n$-form
on $\t{M}$, therefore can be thought of as a closed $n$-form on $M$. Then we define
$${\rm vol}_G(M,\rho)=\int_MD_{\rho}^{\ast}(\omega_X)=\sum_{i=1}^s\epsilon_i {\rm vol}_X(D_{\rho}(\t\Delta_i))$$
Here $\{\Delta_1,\ldots,\Delta_s\}$ are the $n$-simplices of $\Delta_M$,
$\t \Delta_i$ is a lift of $\Delta_i$ and $\epsilon_i=\pm 1$ depends on whether $D_{\rho}|\t\Delta_i$ preserves the orientation or not.

\section{Positive simplicial volume implies virtually positive Seifert volume}\label{Sec-Theorem-virt-positive}

In this section, we adapt Theorem \ref{virt-dom} of \cite{Sun}
to the framework of \cite{DLW} to prove Theorem \ref{virt-positive}.

\subsection{Virtual representation through geometric decomposition}
We recall some results from \cite{DLW}.
The following additivity principle allows
us to compute the representation volume by information on the JSJ pieces.
It is proved by using the
relation between the representation volume and the Chern--Simons theory.

\begin{theorem}[{Additivity principle, \cite[Theorem 3.5]{DLW}; see also \cite{DW2}}]\label{additive}
            Let $M$ be an oriented closed $3$-manifold with JSJ tori $T_1,\cdots,T_r$ and JSJ pieces
            $J_1,\cdots,J_k$,
            and let $\zeta_1,\cdots,\zeta_r$ be slopes on $T_1,\cdots, T_r$, respectively.

            Suppose that $G$ is either ${\rm Iso}_e\t{{\rm SL}_2(\R)}$ or ${\rm PSL}(2;{\C})$, and that
                $$\rho:\pi_1(M)\to G$$
            is a representation vanishing on the slopes $\zeta_i$,  and that
                $\hat\rho_i:\pi_1(\hat{J_i})\to G$
            are the induced representations, where $\hat{J}_i$ is the Dehn filling of $J_i$ along
            slopes adjacent to its boundary, with the induced orientations. Then:
			$${\rm vol}_G(M,\rho)={\rm vol}_G({\hat J}_1,\hat{\rho}_1)+ {\rm vol}_G({\hat J}_2,\hat{\rho}_2)+
			\ldots+ {\rm vol}_G({\hat J}_k,\hat{\rho}_k).$$
\end{theorem}

The following simple lemma 
suggests that we should focus on those JSJ pieces whose groups have non-elementary images under $\rho$.

\begin{lemma}[{\cite[Lemma 3.6]{DLW}}]\label{Almost trivial representation}
Suppose that $G$ is either
${\rm Iso}_e\t{{\rm SL}_2(\R)}$ or ${\rm PSL}(2;{\C})$ and that $M$ is a closed oriented 3-manifold. If $\rho:\pi_1M\to G$ has image either infinite cyclic or
finite, then  ${\rm vol}_G({ M}, {\rho})=0$.
\end{lemma}

The existence of a class inversion for the target group
				has played  an important role in \cite{DLW}
				for constructing virtual representation of mixed $3$-manifold groups.
				Here we quote the following definition.
				An intimately related notion called CI completion
				is introduced and studied in this paper when we prove
				Theorem \ref{main-unbounded} (see Subsection \ref{Subsec-CIcompletion}).

\begin{definition}[{\cite[Definition 5.1]{DLW}}]\label{classInvertible} Let $\mathscr{G}$ be an arbitrary group
        and  $\{\,[A_i]\,\}_{i\in I}$ be a collection
        of conjugacy classes of abelian subgroups.
        By a \emph{class inversion} with respect to $\{\,[A_i]\,\}_{i\in I}$,
        we mean an outer automorphism
        $[\nu]\in\mathrm{Out}(\mathscr{G})$,
        such that for any representative abelian subgroup $A_i$
        of each $[A_i]$, there is a representative automorphism
        $\nu_{A_i}:\mathscr{G}\to\mathscr{G}$ of $[\nu]$
        that preserves $A_i$, taking every
        $a\in A_i$ to its inverse.
        We say $\mathscr{G}$ is \emph{class invertible}
        with respect to $\{\,[A_i]\}_{i\in I}$,
        if there exists a class inversion.
        We often ambiguously call any collection
        of representative abelian subgroups $\{\,A_i\,\}_{i\in I}$
        a class invertible collection, and call any representative automorphism
        $\nu$ a class inversion.
    \end{definition}

Now we state the following fundamental construction about virtual representation extensions.
It uses works of Przytycki--Wise \cite{PW1,PW2,PW3} (and \cite{Wi,HW}) and Rubinstein--Wang \cite{RW}
(see also \cite{Liu})
to understand virtual properties of $3$-manifolds with nontrivial geometric decomposition.

\begin{theorem}[{\cite[Theorem 5.2]{DLW}}]\label{virtualExtensionRep}
        Let $\mathscr{G}$ be a
        group, and
        $M$ be an irreducible orientable closed $3$-manifold with non-trivial JSJ decomposition.
        For a fixed JSJ piece $J_0\subset M$,
        suppose a representation
            $$\rho_0:\pi_1(J_0)\to\mathscr{G}$$
        satisfies the following:
        \begin{itemize}
        \item for every boundary torus $T\subset \partial J_0$, $\rho_0$ has nontrivial kernel
        restricted to $\pi_1(T)$; and
        \item for all boundary tori $T\subset\partial J_0$, $\rho_0(\pi_1(T))$
        form a class invertible collection of abelian subgroups of $\mathscr{G}$.
        \end{itemize}
        Then there exist a finite
        regular cover
            $$\kappa:\tilde{M}\to M,$$
        and a representation
            $$\tilde{\rho}:\pi_1(\tilde{M})\to\mathscr{G},$$
        satisfying the following:
        \begin{itemize}
        \item for one or more elevations $\tilde{J_0}$ of $J_0$,
        the restriction of $\t \rho$ to $\pi_1(\tilde{J_0})$ is,
        up to a class inversion, conjugate
        to the pull-back $\kappa^*(\rho_0)$; and
        \item for any elevation $\tilde{J}$ other than the above, of any geometric piece
        $J$, the restriction of $\tilde{\rho}$ to $\pi_1(\tilde{J})$ is cyclic,
        possibly trivial.
        \end{itemize}
    \end{theorem}

\subsection{Proof of Theorem \ref{virt-positive}}

Now we are ready to prove Theorem \ref{virt-positive}, and here is a sketch of the strategy. Since we can suppose that the manifold has a hyperbolic JSJ piece,
Theorem \ref{virt-dom} gives a virtual representation of the hyperbolic piece with positive Seifert volume. Then
with Lemma \ref{class inversion}, Theorem \ref{virtualExtensionRep} extends the virtual representation to the whole manifold, and  the volume of the virtual representation can be calculated by Theorem \ref{additive} and
Lemma \ref{Almost trivial representation}.


By Theorem \ref{non-zero and zero} and Theorem \ref{virt-dom}, we may assume that $M$ has non-trivial JSJ decomposition
and contains a hyperbolic JSJ piece $J$ in $M$.
Suppose $\partial J$ is a union of tori $T_1,..., T_k$. Let $\alpha_i$ be a slope on $T_i$, then call $\alpha=\{\alpha_1,..., \alpha_k\}$
a slope on $\partial M$. Denote by $J(\alpha)$ the closed orientable 3-manifold obtained by Dehn filling of $k$ solid tori $S_1,..., S_k$
to $J$ along $\alpha$. We can choose $\alpha$ so that  $J(\alpha)$ is a hyperbolic 3-manifold (Theorem \ref{surgered}).

Take a closed orientable manifold $N$ of non-vanishing Seifert volume.
	For example, a circle bundle $N$ with Euler class $e\neq0$ over a closed surface of Euler characteristic $\chi<0$
	works: In fact, for such $N$,
		$$\SV(N)\,=\,\frac{4\pi^2|\chi|^2}{|e|}\,>\,0.$$

By Theorem \ref{virt-dom} there is a finite cover $q: Q\to J(\alpha)$ such that $Q$ dominates $N$, therefore
$\SV(Q)>0$. Let $S=\cup S_i$, then $S'=q^{-1}(S)\subset Q$ is a union of
solid tori and $J'=Q\setminus S'$ is a connected 3-manifold which covers $J$. Moreover, $Q$ is obtained by Dehn filling $S'$ to
$J'$ along $\alpha'$, where $\alpha'$ is a slope of $\partial J'$ which covers $\alpha$
(i.e. each component of $\alpha'$ is an elevation of a component of $\alpha$). Fix $J'$
and $\alpha'$ for the moment, we have the following statement.

	\begin{lemma}
		Let $\t J$ be a finite covering of $J'$ and $\t \alpha$ be the slope of $\partial \t J$ which covers $\alpha'$,
		then $\SV(\t J(\t \alpha))> 0$.
  \end{lemma}

  \begin{proof} By the choice of $\t \alpha$, the covering $\t J\to J'$ extends to a branched covering  $\t J(\t \alpha)\to J'(\alpha')$
  branched over the cores of $S'$. Since the branched covering is a map of non-zero degree, and $\SV(J'(\alpha'))>0$,
  we have $\SV(\t J(\t \alpha))> 0$.
  \end{proof}

  According to  \cite[Proposition 4.2]{DLW},
  there is a finite cover  $p: \t M\to M$ such that each JSJ piece $\t J$ of $\t M$ that covers $J$ factors through
  $J'$. In particular, in the notations we have just used, $\SV(\t J(\t \alpha))>0$.
 To simplify the notations, we still rewrite $\t M$, $\t J$, $\t \alpha$ as $ M$, $J$, $\alpha$.
 Since Theorem \ref{virt-positive} concludes with a virtual property, we need only to prove the following statement.

   \begin{theorem}\label{2}
  Suppose $M$ is a closed orientable 3-manifold with non-trivial  JSJ decomposition,
  and there is a JSJ piece $J$ and a slope $\alpha$ of $\partial J$ such that  $\SV(J( \alpha))> 0$.
  Then there is a finite cover $\t M$ of $M$
  such that $ \SV(\t M)>0$.
  \end{theorem}

  We are going to apply Theorem \ref{virtualExtensionRep} to prove Theorem \ref{2}.
  So we first need to check that the 3-manifold $M$ and the local representation
  $\rho\co \pi_1(J)\rightarrow G$
	(which gives positive representation volume for $J(\alpha)$) in
  Theorem \ref{2} meet the two conditions of Theorem \ref{virtualExtensionRep}.

We first write a presentation of $\pi_1(J(\alpha))$ from $\pi_1(J)$ by attaching $k$ relations from Dehn fillings.
	Let $G=\mathrm{Iso_e\widetilde {SL_2(\R)}}$ be  the identity
	component of  $\mathrm{Iso\widetilde {SL_2(\R)}}$,  the isometry group of
  the Seifert space $\widetilde {\mathrm{SL_2(\R)}}$.
  Then the condition  $\SV(J( \alpha))> 0$ implies that there is a representation $\rho : \pi_1(J) \to G$ such that for each component
  $T_i$ of $\partial J$, $\rho(\pi_1(T_i))$ is a (possibly trivial) cyclic group. Moreover, $\rho$ extends to
  $\hat \rho: \pi_1(J(\alpha)) \to G$ such that $V_G(J(\alpha), \hat \rho)>0$. So the first condition of Theorem 5.2 of \cite{DLW} is satisfied.
  The following lemma, which strengthens \cite[Lemma 6.1 (2)]{DLW}, implies that the
  second condition of Theorem 5.2 of \cite{DLW} is also satisfied.


  \begin{lemma}\label{class inversion}
			$\mathrm{Iso}_e\widetilde{\mathrm{SL}_2(\R)}$ is class invertible with
      respect to all its cyclic subgroups, and a class inversion can
      be realized by an inner automorphism of $\mathrm{Iso}\,\widetilde{{\rm SL}_2(\R)}$.
      The corresponding action on $\widetilde{{\mathrm{SL}}_2(\R)}$ preserves the orientation.
		\end{lemma}

    \begin{proof}
					 There are short exact sequences of groups
							$$0\longrightarrow \R\longrightarrow \mathrm{Iso}\,\widetilde{\mathrm{SL}_2(\R)}
							\stackrel{p}\longrightarrow \mathrm{Iso}\,\Hi^2
							\longrightarrow 1$$ and
							$$0\longrightarrow \R\longrightarrow \mathrm{Iso_e}\,\widetilde{\mathrm{SL}_2(\R)}
							\stackrel{p}\longrightarrow \mathrm{Iso_+}\,\Hi^2
							\longrightarrow 1.$$
						Recall that there are no orientation reversing isometries in
						the $\widetilde{\mathrm{SL}_2(\R)}$ geometry.

						For each element $\nu$ in the component of
						$\mathrm{Iso} \widetilde{\mathrm{SL}_2(\R)}$ not containing the identity,
                        $\nu$ reverses the orientation of $\R$ (the center of $\mathrm{Iso}_e\widetilde{\mathrm{SL}_2(\R)}$).
						So $\nu r \nu^{-1}=r^{-1}$
						for any $r\in \R$, and
						$\mathrm{Iso_e\widetilde{ SL_2(\R)}}$ is class invertible with
                        respect to its center $\R$, and a class inversion can
                        be realized by an inner automorphism of $\mathrm{ Iso\widetilde{{SL}_2(\R)}}$,
                        the corresponding action on $\widetilde{\mathrm{SL}_2(\R)}$ preserves the orientation.
                        Actually, this part of the proof is same with the proof of \cite[Lemma 6.1(ii)]{DLW}).
                         				
                        In the following, we suppose that $\langle \alpha \rangle$ is a cyclic subgroup
						of $\mathrm{Iso}_e\widetilde{\mathrm{SL}_2(\R)}$
						generated by an non-central element $\alpha$.
						
						For each non-trivial element $a$ in $ \mathrm{Iso_+}\Hi^2$, it is straightforward to see
						that there exists a reflection about a geodesic $l_a$ in $\Hi^2$ that conjugats $a$ to its inverse.
						The $l_a$ can be chosen as: (i) passing through the rotation center when
						$a$ is elliptic, (ii)  perpendicular with the axis of $a$ when $a$ is hyperbolic, (iii) passing through
                        the fixed point when $a$ is parabolic.
						
					    By the discussion in the last paragraph and the exact sequences,
					    there exists an element $\nu\in\mathrm{Iso}\widetilde{\mathrm{SL}_2(\R)}$
						lying in the component of $\mathrm{Iso}\widetilde{\mathrm{SL}_2(\R)}$
                        not containing the identity, such that $p(\nu)$ is a reflection of
						$\Hi^2$ conjugating $p(\alpha)$ to its inverse,
						namely, $p(\nu^{-1}\alpha\nu)=p(\alpha^{-1})$.
						We claim that
							$$\nu^{-1} \alpha\nu=\alpha^{-1}.$$

						In fact, by the short exact sequences above, we have that $\nu^{-1} a\nu=a^{-1}r$
						for some $r$ in the center $\R$.
						Since $p(\nu)$ is a reflection, $\nu^2$ is central, so
						$$a=\nu^{-2}a\nu^{2}=\nu^{-1}(a^{-1}r)\nu=(\nu^{-1} r\nu)(\nu^{-1} a\nu)^{-1}
						=r^{-1}(a^{-1}r)^{-1}=ar^{-2}.$$ Here we used that fact that
						$\nu$ is a class inversion for $\langle r\rangle$.
						So $r^{-2}$ is trivial, and $r$ is trivial as
						the center is torsion free.
						This verifies the claim. We conclude that $\nu$ realizes a class inversion
						of the cyclic subgroup $\langle a\rangle$ of $\mathrm{Iso}_e\widetilde{\mathrm{SL}_2(\R)}$.

						Suppose we have two elements $\alpha_1,\alpha_2 \in \mathrm{Iso}_e\widetilde{\mathrm{SL}_2(\R)}$.
						Then there exist $\nu_1, \nu_2 \in \mathrm{Iso} \widetilde{\mathrm{SL}_2(\R)}$
						such that $\nu_i^{-1} \alpha_i\nu_i=\alpha_i^{-1}$, and $p(\nu_i)$ is a reflection of $\Hi^2$,
                        for $i=1,2$. Since any two reflections on $\Hi^2$ are conjugate with each other by an element
						$b\in \mathrm{Iso_+}\,\Hi^2$, we have $p(\beta)p(\nu_1)p(\beta^{-1})=p(\nu_2)$
						for some $\beta\in \mathrm{Iso}_e\widetilde{\mathrm{SL}_2(\R)}$. Therefore we have
						$\beta\nu_1\beta^{-1}=\nu_2r$ for some $r\in \R$.

                        Clearly $\nu_2r$ plays the same role as $\nu_2$ did,
						so $\nu_1$ and $\nu_2$ can be chosen to be conjugate with each other by an element in
						$\mathrm{Iso}_e   \widetilde{\mathrm{SL}_2(\R)}$. We have verified that $\mathrm{Iso}_e\widetilde{\mathrm{SL}_2(\R)}$
                        is class invertible with respect to all its cyclic subgroups, and a class inversion can
                        be realized by an inner automorphism of $\mathrm{Iso}\,\widetilde{{\rm SL}_2(\R)}$,
                        and the corresponding action on $\widetilde{{\mathrm{SL}}_2(\R)}$ preserves the orientation.
				\end{proof}

        \begin{proof}[{Proof of Theorem \ref{2}}]
            Fix $J$, $\alpha$, and $\rho:\pi_1(J)\to G$ as our previous discussion, and denote them by $J_0$,
            $\alpha_0$, and $\rho_0$ to match the notations of Theorem \ref{virtualExtensionRep}. Since
            $\rho_0:\pi_1(J_0)\to G$ meets the two conditions of Theorem \ref{virtualExtensionRep},
            we can virtually extend $\rho_0$ to $\t \rho:\pi_1(\t M)\to G$, which satisfies those conditions
            in the conclusion of Theorem \ref{virtualExtensionRep}.

            By the additivity principle (Theorem \ref{additive}), we need only to compute the representation
            volume for each JSJ piece of $\t M$, then add the volumes together to compute
            $V_G(\t M, \t \rho)$. By Theorem \ref{virtualExtensionRep} and Lemma \ref{Almost trivial representation},
            only those elevations $\t{J_0}$ of $J_0$ such that the restriction of $\t {\rho}$ to $\pi_1(\t {J_0})$
            is conjugate to the pull-back $\kappa^*(\rho_0)$, up to a class inversion, could contribute to the Seifert
            representation volume of $\t{M}$.

            By Lemma \ref{class inversion}, the class inversions can be realized by conjugations of orientation
            preserving isomorphisms of $\widetilde{{\mathrm{SL}}_2(\R)}$, therefore the volumes of
            all these elevations are  positive multiples of  $V_G(J_0(\alpha_0), \hat \rho_0)>0$.
            So the Seifert representation volume of $\t M$ with respect to $\t \rho$ is positive, which implies
            $\SV(\tilde M)>0$.
            \end{proof}
						
	The completion of the proof of Theorem \ref{2} also completes the proof of Theorem \ref{virt-positive}.

  \begin{remark}\label{reformulation}
	We can reformulate what we have done in this section by the following proposition:

	\begin{proposition}\label{refine-virtualSeifertRepresentation}
		Let $M$ be an orientable closed mixed $3$-manifold and $J_0$ be a distinguished hyperbolic JSJ piece of $M$.
		Suppose that $\widehat{J}_0$ is
		a closed hyperbolic Dehn filling of $J_0$
		by sufficiently long boundary slopes.
		
		\begin{enumerate}
		\item
		For any finite cover $\widehat{J}'_0$ of $\widehat{J}_0$ and any representation
			$$\eta:\pi_1(\widehat{J}'_0)\longrightarrow\mathrm{Iso}_e\Sft,$$
		there exist a finite cover
			$$\tilde{M}'\longrightarrow M$$
		and a representation
			$$\rho:\pi_1(\tilde{M}')\longrightarrow\mathrm{Iso}_e\Sft,$$
		with the following properties:
    \begin{itemize}
        \item For one or more elevations $\tilde{J}'$ of $J_0$ contained in $\tilde{M}'$,
        the covering $\tilde{J}'\to J_0$ factors through
				a covering $\tilde{J}'\to J'_0$, where $J'_0\subset \widehat{J}'_0$
				denotes the unique elevation of $J_0\subset \widehat{J}_0$.
				The restriction of $\rho$ to $\pi_1(\tilde{J}')$ is conjugate to
				either the pull-back $\beta^*(\eta)$ or the pull-back $\beta^*(\nu\eta)$,
				where $\nu$ is a class inversion, and $\beta$ is the composition of the mentioned maps:
					$$\tilde{J}'\stackrel{\mathrm{cov.}}\longrightarrow J'_0\stackrel{\mathrm{fill}}\longrightarrow \widehat{J}'_0;$$
        \item For any elevation $\tilde{J}'$ other than the above, of any JSJ piece
        $J$ of $M$, the restriction of $\rho$ to $\pi_1(\tilde{J})$ has cyclic image,
        possibly trivial.
    \end{itemize}
		\item ${\mathrm{Vol}_{\mathrm{Iso}_e\Sft}(\tilde{M}',\rho)}$ is a positive multiple of ${\mathrm{Vol}_{\mathrm{Iso}_e\Sft}(\widehat{J}'_0,\eta)}.$
		
		\end{enumerate}
	\end{proposition}
	
	The first part of Proposition \ref{refine-virtualSeifertRepresentation}
	is a specialized refined statement of Theorem \ref{virtualExtensionRep};
	the second part supplies a slot to connect with Theorem \ref{virt-dom}.
	Therefore, Theorem \ref{virt-positive} is a consequence of Proposition \ref{refine-virtualSeifertRepresentation}
	and Theorem \ref{virt-dom}.	The stronger result, Theorem \ref{main-unbounded}, will follow from
	an efficient version of this proposition (Theorem \ref{controlledvirtualSeifertRepresentation})
	plus the efficient virtual domination (Theorem \ref{efficient-virt-domi}).
	%
	\end{remark}

\section{Efficient virtual domination by hyperbolic 3-manifolds}\label{Sec-Theorem-efficient-virt-domi}

In this section, we employ the work of A.~Gaifullin \cite{Ga}
to derive Theorem \ref{efficient-virt-domi} from Theorem \ref{virt-dom}.
We quote the statement below for convenience.

\begin{theorem*}[\ref{efficient-virt-domi}]
	For any closed oriented hyperbolic $3$--manifold $M$, there exists a positive constant $c(M)$ such that the following statement holds. For any closed
	oriented $3$--manifold $N$ and any $\epsilon>0$, there exists a finite cover $M'$ of $M$ which admits a non--zero degree map $f:M'\rightarrow N$, such that
		$$\|M'\|\leq c(M)\cdot|\mathrm{deg}(f)|\cdot(\|N\|+\epsilon).$$
\end{theorem*}

\begin{remark}
	In fact, the same statement holds for any closed orientable manifold
	which virtually dominates all closed orientable manifolds of the same dimension.
	For dimension $3$, all hyperbolic manifolds have such property \cite{Sun}.
	For any arbitrary dimension,
	manifolds with this property have been discovered by
	A.~Gaifullin \cite{Ga}.
	The $3$-dimension example $M_{\Pi^3}$ of Gaifullin is not a hyperbolic manifold,
	but we point out that a constant $c_0=24v_8\,/\,v_3\approx 86.64$
	is sufficient for this case, where $v_8$ is the volume of the ideal regular hyperbolic octahedron
	and $v_3$ is the volume of the ideal regular hyperbolic tetrahedron.
\end{remark}
	
	\subsection{URC manifolds}
	As introduced by A.~Gaifullin \cite{Ga}, a closed orientable (topological) $n$-manifold $M$ is said to have
	the property of \emph{Universal Realisation of Cycles} (\emph{URC})
	if every homology class of $H_n(X;\,\Z)$
	of an arbitrary topological space $X$ has a positive integral multiple
	which can be realized by the fundamental class of a finite cover $M'$ of $M$, via a map $f:M'\to X$.
	
	For any arbitrary dimension $n$, Gaifullin shows that examples of URC $n$-manifolds
	can be obtained by taking some $2^n$-sheeted cover
		$$M_{\Pi^n}$$
	of some $n$-dimensional orbifold $\Pi^n$. More precisely,
	the underlying topology space of $\Pi^n$ is
	the \emph{permutahedron}, namely, the polyhedron combinatorially isomorphic to
	the convex hull of the points $(\sigma(1),\cdots,\sigma(n+1))$
	of $\R^{n+1}$, where $\sigma$ runs over all permutations of $\{1,\dots,n+1\}$.
	The orbifold structure of $\Pi^n$ is given
	so that each codimension-$1$ face is a reflection wall,
	so each codimension-$k$ face is the local fixed point set of a $\Z_2^k$--action.
	The abelian characteristic cover
	of $\Pi^n$ on which $H_1(\Pi^n;\Z_2)\cong \Z_2^n$
	acts is the orientable closed $n$-manifold $M_{\Pi^n}$.
	In particular, $M_{\Pi^n}$ can be obtained by facet pairing of $2^n$ permutahedra.
	
	The following quantitative version of Gaifullin's proof \cite[Section 5]{Ga}
	is important for our application.
	Recall that a (compact) \emph{pseudo $n$-manifold} is a finite simplicial complex of which each simplex
	is contained in an $n$-simplex and each $(n-1)$-simplex is contained in exactly two $n$-simplices.
	Topologically, a pseudo $n$-manifold is just a manifold away from its codimension-$2$ skeleton.
	A \emph{strongly connected orientable} pseudo $n$-manifold means that, away from the codimension-$2$ skeleton,
	the manifold is connected and orientable,
	or equivalently, that the $n$-dimensional integral homology is isomorphic to $\Z$.
	In particular, the concept of (unsigned) mapping degree can be extended similarly to maps
	between strongly connected orientable pseudo $n$-manifolds.
	
	\begin{theorem}[{See \cite[Proposition 5.3]{Ga}}]\label{Ga-quantitative}
		For any strongly connected orientable pseudo $n$-manifold $Z$, there exists a finite cover
		$M'_{\Pi^n}$ of $M_{\Pi^n}$ and a non-zero degree map $f_1:M'_{\Pi^n}\to Z$, such that
			$$\#\{n\textrm{-permutahedra of }M'_{\Pi^n}\}\,=\,
			(n+1)!\cdot|\mathrm{deg}(f_1)|\cdot\#\{n\textrm{-simplices of }Z\}.$$
	\end{theorem}
	
	\begin{remark}
		The map $f_1$ is as asserted by \cite[Proposition 5.3]{Ga}.
		The cover $\widehat{M}_{\Pi^n}=\mathcal{U}_{\Pi^n}/\Gamma_H$ there
		is rewritten as $M'_{\Pi^n}$ in our paraphrase.
		To compare with the statement of \cite[Proposition 5.3]{Ga},
		the index $|W:\Gamma_H|$ there equals the number of permutahedra in $M'_{\Pi^n}$ here;
		the notation $|A|$ there stands for the number of $n$-simplices in the barycentric subdivision
		of $Z$, which equals $(n+1)!$ times the number of $n$-simplices of $Z$ here. For dimension $3$, all orientable closed hyperbolic manifold are known to be
        URC \cite{Sun}.
	\end{remark}		
		
	\subsection{Virtual domination through URC $3$-manifolds}
	We combine the results of \cite{Ga} and \cite{Sun} to prove Theorem \ref{efficient-virt-domi}.
	The following lemma allows us to create an efficient virtual realization of the fundamental class of $N$.
	


\begin{lemma}\label{simplicial}
For any closed oriented $n$--manifold $N$ and any $\epsilon>0$, there exists a connected oriented pseudo $n$--manifold $Z$ and a non--zero degree map
$f:Z\rightarrow N$, such that $$\#\{n\textrm{--simplices of }Z\}\leq |\mathrm{deg}(f)|\cdot (\|N\|+\epsilon).$$
\end{lemma}

\begin{proof}
By the definition of the simplicial volume, for any $\epsilon>0$, there exists a singular cycle $$\alpha=\sum_{i=1}^k s_i\sigma_i\in Z_n(N,\mathbb{R})$$
such that $[\alpha]=[N]\in H_n(N,\mathbb{R})$ and $$\sum_{i=1}^k |s_i|<\|N\|+\epsilon.$$ Here $s_i$ are real numbers and $\sigma_i$ are maps from the
standard oriented $n$-simplex to $N$.

Since
$$\left\{
\begin{array}{ccc}
\sum_{i=1}^k x_i\sigma_i & \in & Z_n(N,\mathbb{R})\\
{\big [}\sum_{i=1}^k x_i\sigma_i{\big ]} &= &{\big [}N {\big ]}\in H_n(N,\mathbb{R})
\end{array}
\right. $$
can be expressed as linear equations with integer coefficients, it has a rational solution $(r_1,\cdots,r_k)$ close to $(s_1,\cdots,s_k)$ such that $r_i\in \mathbb{Q}$
and $$\sum_{i=1}^k |r_i|<\|N\|+\epsilon.$$ In particular, $[\sum_{i=1}^k r_i\sigma_i] =[N]\in H_n(N,\mathbb{R})$ holds. Here we can suppose that each $r_i$ is
--negative, by reversing the orientation of $\sigma_i$ if necessary.

Let the least common multiple of the denominators of $r_i$ be denoted $m$, then
$$\beta=m(\sum_{i=1}^k r_i\sigma_i)=\sum_{i=1}^k (mr_i)\sigma_i\in Z_n(N;\mathbb{Z})$$ is an
integer linear combination of $\sigma_i$ and $[\beta]=m[N]\in H_n(N,\mathbb{R})$.

Here we can think that we have $mr_i$ copies of the standard oriented $n$--simplex that is mapped as $\sigma_i$, $i=1,2,\cdots,k$. The condition that
$\sum_{i=1}^k (mr_i)\sigma_i$ being an $n$--cycle implies that we can find a pairing of all the $(n-1)$--dimensional faces of the collection of copies
of $\sigma_i$s, such that each such pair are mapped to the same singular $(n-1)$-simplex in $N$, with opposite orientation.

This pairing instructs us to build an oriented pseudo--manifold $Z'$ (possibly disconnected). It is given by taking $\sum_{i=1}^k mr_i$ copies of the standard
oriented $n$--simplex and pasting them together by the pairing given above. Then the singular $n$-simplices $\{\sigma_i\}_{i=1}^k$ induces a map
$f_0:Z'\rightarrow N$.

Let $[Z']$ be the homology class in $H_n(Z')$ which is represented by the $n$--cycle which takes each oriented $n$--simplex in $Z'$ exactly once. It is easy to
see that $f_0([Z'])=[\beta]=m[N]$, so $f_0$ has mapping degree $\text{deg}(f_0)=m$. Moreover, the number of $n$--simplices in $Z'$ is just
$$\sum_{i=1}^k (mr_i)=m(\sum_{i=1}^k r_i)<m(\|N\|+\epsilon)=\text{deg}(f_0)\cdot(\|N\|+\epsilon).$$

If $Z'$ is connected, we are done with the proof. If $Z'$ is disconnected, take the component $Z$ of $Z'$ such that
$$\frac{\text{deg}(f_0|_Z)}{\#\{n\textrm{--simplices of }Z\}}$$ is not smaller than the corresponding number for all the other components of $Z'$.
Then $f=f_0|_Z$ satisfies the desired condition in this lemma.
\end{proof}	
		
	\subsubsection{Construction of $(M',f)$}
	Let $M$ be a closed orientable hyperbolic 3-manifold, and $N$ be any closed orientable $3$-manifold.
	Given any constant $\epsilon>0$, denote by
		$$p:Z\longrightarrow N$$
	a virtual realization of the fundamental class of $N$
	by a strongly connected orientable pseudo $3$-manifold, as guaranteed by Lemma \ref{simplicial}.
	Take a finite cover $M'_{\Pi^3}$ of Gaifullin's URC $3$-manifold $M_{\Pi^3}$ and an efficient domination map
		$$f_1:M'_{\Pi^n}\longrightarrow Z$$
	which come from Theorem \ref{Ga-quantitative}.
	Take a finite cover $\tilde{M}$ of $M$ and a $\pi_1$-surjectively $2$-domination map:
		$$f_2:\tilde{M}\longrightarrow M_{\Pi^3}$$
	which comes from Theorem \ref{virt-dom}.
	Then there exists a unique finite cover $M'$ of $M$, up to isomorphism of covering spaces,
	and a unique $\pi_1$-surjective $2$-domination map $f_2':M'\to \tilde{M}_{\Pi^3}$
	that fit into the following commutative diagram of maps:
	$$\begin{CD}
		M'@> f'_2>> M'_{\Pi^3}\\
		@VVV @VVV\\
		\tilde{M} @>f_2 >> M_{\Pi^3}
	\end{CD}$$
	Indeed, $M'$ is the cover of $\tilde{M}$
	that corresponds to the subgroup $(f_{2\sharp})^{-1}(\pi_1(M'_{\Pi^3}))$ of $\pi_1(\tilde{M})$
	(after choosing some auxiliary basepoints).
	The finite cover $M'$ of $M$ and the composed map:
		$$f:\,M'\stackrel{f'_2}\longrightarrow\tilde{M}_{\Pi^3}\stackrel{f_1}\longrightarrow Z\stackrel{p}\longrightarrow N$$
	are the claimed objects of Theorem \ref{efficient-virt-domi}.
	
	\subsubsection{Verification}
	With the notations above, the commutative diagram above implies:
	$$\frac{\|M'\|}{\|\tilde{M}\|}\,=\,
	[M':\tilde{M}]\,=\,[M'_{\Pi^3}:M_{\Pi^3}]\,=\,
	\frac{\#\{\textrm{permutahedra of }M'_{\Pi^3}\}}{\#\{\textrm{permutahedra of }M_{\Pi^3}\}}.$$
	
	Observe that there are $2^3=8$ permutahedra in Gaifullin's URC $3$-manifold $M_{\Pi^3}$.
	On the other hand, by Theorem \ref{Ga-quantitative} and Lemma \ref{simplicial}, the construction of $M'_{\Pi^3}$ and $Z$ yields:
	\begin{eqnarray*}
	\#\{\textrm{permutahedra of }M'_{\Pi^3}\}
	&=&	4!\cdot|\mathrm{deg}(f_1)|\cdot\#\{\textrm{tetrahetra of }Z\}\\
	&<&	24\cdot|\mathrm{deg}(f_1)|\cdot|\mathrm{deg}(p)|\cdot(\,\|N\|+\epsilon\,)\\
	&=&	\frac{24}{2}\cdot|\mathrm{deg}(f'_2)|\cdot|\mathrm{deg}(f_1)|\cdot|\mathrm{deg}(p)|\cdot(\,\|N\|+\epsilon\,)\\
	&=&	12\cdot|\mathrm{deg}(f)|\cdot(\|N\|+\epsilon).
	\end{eqnarray*}
	Therefore,
	$$\|M'\|\,<\,\frac{12\cdot|\mathrm{deg}(f)|\cdot(\|N\|+\epsilon)\cdot\|\tilde{M}\|}{8}
	\,=\,c_0\cdot|\mathrm{deg}(f)|\cdot(\|N\|+\epsilon).$$
	where the constant $c_0$ is taken to be:
		$$c_0\,=\,\frac32\cdot\|\tilde{M}\|.$$
	Note that the constant $c_0>0$ depends only on the hyperbolic $3$-manifold $M$,
	because $\tilde{M}$ is constructed by Theorem \ref{virt-dom} without
	referring to $N$ or $\epsilon$.
	
	This completes the proof of Theorem \ref{efficient-virt-domi}.

\subsection{Virtual Seifert volume of closed hyperbolic $3$--manifolds}


We have mentioned in the introduction  that Theorem \ref{virt-positive} for closed hyperbolic 3-manifolds follows directly
from Theorem \ref{virt-dom}. Similarly, Theorem \ref{main-unbounded} for hyperbolic closed 3-manifolds is a corollary of Theorem \ref{efficient-virt-domi}.

\begin{corollary}\label{hyper-unbounded}
For any closed oriented hyperbolic $3$--manifold $M$, the set of values
$$\left\{\frac{\SV(M')}{[M':M]}\,:\,M'\textrm{ any finite cover of }M\right\}$$
is not bounded.
\end{corollary}

\begin{proof}
Take a closed orientable manifold $N$ of non-vanishing Seifert volume
	and vanishing simplicial volume. For example, a circle bundle $N$ with
	Euler class $e\neq0$ over a closed surface of Euler characteristic $\chi<0$ works.

For every positive integer $n$, apply Theorem \ref{efficient-virt-domi}
with $\epsilon=1/n$. There exists a finite cover $M_n\rightarrow M$ and a non-zero
degree map $f_n:M_n\rightarrow N$ such that

$$\|M\|\cdot[M_n:M]=\|M_n\|\leq c(M)\cdot |\text{deg}(f_n)|\cdot(||N||+\frac{1}{n})=c(M)\cdot |\text{deg}(f_n)|\cdot\frac{1}{n}.$$
So we have  $$[M_n:M]\leq \frac{c(M)\cdot |\text{deg}(f_n)|\cdot\frac{1}{n}}{||M||}.$$

Since $\SV(M_n)\geq |\text{deg}(f_n)|\cdot \SV(N)$, we have
$$\frac{\SV(M_n)}{[M_n:M]}\geq \frac{|\text{deg}(f_n)|\cdot{\SV(N)}}{\frac{c(M)\cdot |\text{deg} (f_n)|\cdot\frac{1}{n}}{\|M\|}}
\,=\,n\cdot\frac{\|M\|\cdot \SV(N)}{c(M)}.$$

Since $K={\|M\|\cdot \SV(N)}/{c(M)}$ is a positive constant,
$\{\,{\SV(M_n)}/{[M_n:M]}\,\}$ is not a bounded sequence,
so we are done.
\end{proof}

\section{Positive simplicial volume implies unbounded virtual Seifert volume}\label{Sec-Theorem-main-unbounded}
	In this section, we prove Theorem \ref{main-unbounded} following the strategy of
	the proof of Theorem \ref{virt-positive} summarized in Remark \ref{reformulation}.
	The main body of the proof is the following theorem
	which produces virtual Seifert representations with controlled volume,
	(compare Proposition \ref{refine-virtualSeifertRepresentation}).
	
%
%
				%
					
	\begin{theorem}\label{controlledvirtualSeifertRepresentation}
		Let $M$ be an orientable closed mixed $3$-manifold and $J_0$ be a distinguished hyperbolic JSJ piece of $M$.
		Suppose that $\widehat{J}_0$ is
		a closed hyperbolic Dehn filling of $J_0$
		by sufficiently long boundary slopes.
		
		\begin{enumerate}
		\item
		For any finite cover $\widehat{J}'_0$ of $\widehat{J}_0$ and any representation
			$$\eta:\pi_1(\widehat{J}'_0)\longrightarrow\mathrm{Iso}_e\Sft,$$
		there exist a finite cover
			$$\tilde{M}'\longrightarrow M$$
		and a representation
			$$\rho:\pi_1(\tilde{M}')\longrightarrow\mathrm{Iso}_e\Sft,$$
		with the following properties:
    \begin{itemize}
        \item For one or more elevations $\tilde{J}'$ of $J_0$ contained in $\tilde{M}'$,
        the covering $\tilde{J}'\to J_0$ factors through
				a covering $\tilde{J}'\to J'_0$, where $J'_0\subset \widehat{J}'_0$
				denotes the unique elevation of $J_0\subset \widehat{J}_0$.
				The restriction of $\rho$ to $\pi_1(\tilde{J}')$ is conjugate to
				either the pull-back $\beta^*(\eta)$ or the pull-back $\beta^*(\nu\eta)$,
				where $\beta$ is the composition of the mentioned maps:
					$$\tilde{J}'\stackrel{\mathrm{cov.}}\longrightarrow J'_0\stackrel{\mathrm{fill}}\longrightarrow \widehat{J}'_0;$$
        \item For any elevation $\tilde{J}'$ other than the above, of any JSJ piece
        $J$ of $M$, the restriction of $\rho$ to $\pi_1(\tilde{J})$ has cyclic image,
        possibly trivial.
    \end{itemize}
		\item Furthermore, there exists a positive constant $\alpha_0$ depending only on $M$ and
		the Dehn filling $J_0\to \widehat{J}_0$, such that for any $\widehat{J}'_0$ and $\eta$ above,
		the asserted $\tilde{M}'$ and $\rho$ can be constructed so that
		the sum of the covering degrees $[\tilde{J}':J_0]$ over all the elevations $\tilde{J}'$
		of the $\beta$-pull-back type equals $\alpha_0\cdot[\tilde{M}':M]$. Therefore,		
		$$\frac{\mathrm{Vol}_{\mathrm{Iso}_e\Sft}(\tilde{M}',\rho)}{[\tilde{M}':M]}\,= \,\alpha_0\cdot\frac{\mathrm{Vol}_{\mathrm{Iso}_e\Sft}(\widehat{J}'_0,\eta)}{[\widehat		 {J}'_0:\widehat{J}_0]}.$$
	\end{enumerate}
\end{theorem}
	
	The rest of this section is devoted to the proof of Theorem \ref{controlledvirtualSeifertRepresentation},
	before which we derive Theorem \ref{main-unbounded}
	from Theorem \ref{controlledvirtualSeifertRepresentation} and Corollary \ref{hyper-unbounded}.

\subsection{Proof of Theorem \ref{main-unbounded}}				
	Since we have proved Theorem \ref{main-unbounded} for hyperbolic 3-manifolds (Corollary \ref{hyper-unbounded}),
	we may assume that $M$ is non-geometric with at least one hyperbolic piece, or in other words, \emph{mixed}.
	The mixed case is derived from the hyperbolic case and Theorem \ref{controlledvirtualSeifertRepresentation}.
		
	Take a hyperbolic piece $J$ of $M$ and let $\widehat{J}$ be a closed hyperbolic Dehn filling of $J$.
  By Corollary \ref{hyper-unbounded}, there are finite covers  $\{\widehat{J}'_{n}\}$
	of $\widehat{J}$ such that
	$$\frac{\SV(\widehat{J}'_n)}{[\widehat{J}'_n:\widehat{J}]}\,\geq\, nK$$ for some constant $K>0$.
	 Let
		$$\eta_n:\pi_1(\widehat{J}'_n)\to\mathrm{Iso}_e\Sft$$
		be a representation realizing $\SV(\widehat{J}'_{n})$.
	
	Granting Theorem \ref{controlledvirtualSeifertRepresentation}, 
	there exist finite covers $\tilde{M}'_n$ of $M$ and representations:
		$$\rho_n:\pi_1(\tilde{M}'_n)\to\mathrm{Iso}_e\Sft,$$
	such that
		$$\frac{|\mathrm{Vol}_{\mathrm{Iso}_e\Sft}(\tilde{M}'_n;\rho_n)|}{[\tilde{M}'_n:M]}
		\,=\,\alpha_0\cdot\frac{|\mathrm{Vol}_{\mathrm{Iso}_e\Sft}(\widehat{J}'_n;\eta_n)|}{[\widehat{J}'_n:\widehat{J}]}
		\,=\,\alpha_0\cdot\frac{\SV(\widehat{J}'_n)}{[\widehat{J}'_n:\widehat{J}]},$$
	where the positive constant $\alpha_0$ is determined by $M$ and $J_0\to\widehat{J}_0$.
	Therefore,
		$$
		\frac{\SV(\tilde{M}'_n)}{[\tilde{M}_n':M]}
		\,\geq\,\frac{|\mathrm{Vol}_{\mathrm{Iso}_e\Sft}(\tilde{M}'_n;\rho_n)|}{[\tilde{M}_n':M]}
		\,=\,\alpha_0\cdot\frac{\SV(\widehat{J}'_n)}{[\widehat{J}'_n:\widehat{J}]}
		\,\geq\, n\alpha_0K,$$
	so the sequence $\{\,{\SV(\tilde M_n')}/{[\tilde M'_n:M]}\,\}$ is unbounded.
	This completes the proof of Theorem \ref{main-unbounded}.

\subsection{CI completions of hyperbolic $3$-manifolds}\label{Subsec-CIcompletion}
The statement of Theorem \ref{controlledvirtualSeifertRepresentation} (2) suggests certain
relation between the asserted representation $\rho:\pi_1(\tilde{M}')\longrightarrow\mathrm{Iso}_e\Sft$
and the given representation $\eta:\pi(\widehat{J}_0)\longrightarrow\mathrm{Iso}_e\Sft$.
It would certainly hold if $\rho$ factored through the restriction of $\eta$
to some finite covers of $\widehat{J}_0$.
However, the latter is a much stronger requirement that exceeds our ability.
To overcome this difficulty, we examine the machinery of Theorem \ref{virtualExtensionRep}
and observe that $\rho$ does factor through a finite cover of certain CW complex
associated with $\widehat{J}_0$, which looks like $\widehat{J}_0$ attached with a number of
Klein bottles.
In the following, we formalize the idea and introduce \emph{CI completions},
where CI is a brievation for class inversion.

				In general, given an arbitrary group with a collection of conjugacy classes of abelian subgroups,
				it is possible to embed the group into a larger group which possesses a class inversion
				with respect to the induced collection.
				For concreteness, we only consider the special case of
				CI completions
				for orientable closed hyperbolic $3$-manifolds,
				with respect to a collection of mutually distinct embedded closed geodesics.

				\subsubsection{Construction of the CI completion}\label{Subsec-constructionCI}
				Let $V$ be an orientable closed hyperbolic $3$-manifold,
				and let $\gamma_1,\cdots,\gamma_s$
				be a collection of mutually distinct
				embedded closed geodesics of $V$.
				
				The \emph{CI completion} of $V$ with respect to $\gamma_1,\cdots,\gamma_s$
				is a pair
					$$(W,\sigma_{W})$$
				where $W$ is a specific CW space equipped with a distinguished embedding $V\to W$
				and $\sigma_{W}:W\to W$ is a free involution.
				The construction is as follows.
								
				Take the product space $V \times\Z$
				where $\Z$ is endowed with the discrete topology,
				and for each $\gamma_i$, take a cylinder $L_i$ parametrized as $S^1\times\R$,
				where  $S^1$ is identified with the unit circle of the complex plane $\C$.
				We regard each closed geodesic $\gamma_i$ as a map $S^1\to V$.
				Identify the circles $S^1\times\Z$ of $L_i$
				with closed geodesics of $V\times\Z$
				by taking any point $(z,n)\in S^1\times\Z$
				to either $(\gamma_i(z),n)$ or $(\gamma_i(\bar{z}),n)$ depending on the parity of $n$.
				We agree to use $\gamma_i(z)$ for even $n$ and $\gamma_i(\bar{z})$ for odd $n$.
				The resulting space $\widetilde{W}_\Z$ is equipped with a covering transformation
				$\sigma:\,\widetilde{W}_\Z\to \widetilde{W}_\Z$
				which takes any point $(x,n)\in V\times\Z$
				to $(x,n+1)$, and any point $(z,t)\in L_i$ to $(\bar{z},t+1)$.
				The quotient of $\tilde{W}_\Z$ by the action of $\langle\sigma^2\rangle$
				is a space $W$ with a covering transformation $\sigma_{W}$ induced by $\sigma$.
				
				One may visualize the further quotient space
				$W\,/\,\langle\sigma_{W}\rangle$ as
				a $3$-manifold $V$ with Klein bottles hanging on the closed geodesics
				$\gamma_i$, one on each. Then $W$ is a double cover
				of that space into which $V$ lifts, and on which
				the deck transformation $\sigma_{W}$ acts.
				As a CW space with a free involution, the isomorphism type of
				$(W,\sigma_W)$ is independent of the auxiliary parametrizations
				in the construction, and the isomorphism may further be required to fix
				the distinguished inclusion of $V$.
				
				\subsubsection{Properties of CI completions}
				We study the relation of CI completions with class inversions
				and their behavior under finite covers.
				
				\begin{proposition}\label{propertiesCI}
				Let $V$ be an orientable closed hyperbolic $3$-manifold,
				and let $\gamma_1,\cdots,\gamma_s$
				be a collection of mutually distinct
				embedded closed geodesics of $V$.
				Denote by $(W,\sigma_W)$ the CI completion of $V$
				with respect to $\gamma_1,\cdots,\gamma_s$.
				\begin{enumerate}
						\item The outer automorphism of $\pi_1(W)$
						induced by $\sigma_W$ is a class inversion of $\pi_1(W)$ with respect to the collection
						of conjugacy classes of the maximal cyclic subgroups $\pi_1(\gamma_1),\cdots,\pi_1(\gamma_s)$
						of $\pi_1(W)$
						corresponding to the canonically included free loops;
						\item Suppose that $\mathscr{G}$ is a group which possesses
						a class inversion $[\nu]\in\mathrm{Out}(\mathscr{G})$
						with respect to the conjugacy classes of all the cyclic subgroups.
						Then for any homomorphism $\eta:\pi_1(V)\to \mathscr{G}$
						then there exists an extension of $\eta$ to $\pi_1(W)$:
							$$\eta:\pi_1(W)\to \mathscr{G}.$$
						Moreover, for any representative automorphisms $\sigma_{W\sharp},\nu$ of the outer autormorphisms
						accordingly,
						the image $\eta\sigma_{W\sharp}(\pi_1(V))$ is conjugate to $\nu\eta(\pi_1(V))$ in $\mathscr{G}$.
						\item Suppose that $\kappa:V'\to V$ is a covering map of finite degree.
						Denote by $(W',\sigma_{W'})$ the CI completion of $V'$ with respect to
						all the elevations in $V'$ of $\gamma_1,\cdots,\gamma_s$.
						Then there exists an extension of $\kappa$:
							$$\kappa:W'\to W$$
						which is a covering map
						equivariant under the action of $\sigma_{W'}$ and $\sigma_W$.
						In particular,
						the covering degree is preserved under the extension.
					\end{enumerate}
				\end{proposition}
			
				\begin{proof}
					Recall that $W$ is topologically the union of $V$, $\sigma_W(V)$, and annuli $A_i$, $\sigma_W(A_i)$.
					Each annulus $A_i$ has its boundary attached to $V\sqcup \sigma_W(V)$ in such a way that
					$\gamma_i\subset V$ can be freely homotoped to the orientation-reversal of
					$\sigma_W(\gamma_i)\subset\sigma_W(V)$
					through $A_i$, and the annuli $\sigma_W(A_i)$ make the homotopy as well.
					
					Statement (1) is now obvious from the above description.
					
					Statement (2) can also be seen topologically. To this end, let $X$ be a CW model
					for the Eilenberg--MacLane CW space $K(\mathscr{G},1)$.
					Uniquely up to free homotopy, the outer automorphism $[\nu]$ can be realized by a map $R:X\to X$,
					and the homomorphism $\eta$ can be realized	as a map $f:V\to X$.
					With respect to the inclusion $V\to W$, we define a map $F:W\to X$, which extends $f$, as follows.
					First define the restriction of $F$ to $V$ and $\sigma_W(V)$ are $f$ and $Rf$, respectively.
					Since $\nu$ is a class inversion, each $f\gamma_i$ is freely homotopic to the orientation-reversal
					of $Rf\gamma_i$, as a map $S^1\to X$, so the homotopy defines maps $F|:A_i\to X$ and
					$F|:\sigma_W(A_i)\to X$. The resulting map $F:W\to X$ extends $f:V\to X$, so on the level of
					fundamental groups, it gives rise to the claimed extension of $\eta:\pi_1(V)\to\mathscr{G}$
					over $\pi_1(W)$.
					
					Statement (3) follows from construction on further quotient spaces. Observe that
					the quotient space $W\,/\,\langle\sigma_W\rangle$, rewritten as $\bar{W}$,
					is topologically the union of $V$ and Klein bottles $B_i$, where $B_i$ are projected from $A_i$.
					Then any finite covering map $V'\to V$ gives rise to a covering map of the same degree
					$\bar{W}'\to \bar{W}$. The covering of Klein bottles are induced by the coverings of
					$\gamma_i\subset \bar{W}$
					by their elevations. In fact, the covering $\bar{W}'\to \bar{W}$ is unique up to homotopy.
					The covering $\bar{W}'\to \bar{W}$ induces two equivariant covering maps $W'\to W$,
					differing by deck transformation.
					The one that respects the distinguished inclusions is as claimed.
				\end{proof}

			\subsection{Virtual representations through CI completions}
			With our gadgets of CI completions,
			we invoke Theorem \ref{virtualExtensionRep}
			to derive the asserted virtual representations of Theorem \ref{controlledvirtualSeifertRepresentation}.

			\subsubsection{Construction for the basic level}
		Let $M$ be an orientable closed mixed $3$-manifold and $J_0$ be a distinguished hyperbolic JSJ piece of $M$.
		Suppose that $\widehat{J}_0$ is
		a closed hyperbolic Dehn filling of $J_0$
		by sufficiently long boundary slopes,
		which are denoted by $\gamma_1,\cdots,\gamma_s$.
		Let
			$$(W,\sigma_{W})$$
		be the CI completion of $\widehat{J}_0$ with respect to $\gamma_1,\cdots,\gamma_s$,
		(see Subsection \ref{Subsec-constructionCI}).
		Since $\pi_1(W)$ is class invertible with respect to the conjugacy classes
		of subgroups $\pi_1(\gamma_i)$ (Proposition \ref{propertiesCI} (1)),
		Theorem \ref{virtualExtensionRep} can be applied with the target group $\pi_1(W)$
		and the initial homomorphism:
			$$\pi_1(J_0)\longrightarrow \pi_1(W)$$
		induced by the composition of the Dehn filling inclusion $J_0\subset \widehat{J}_0$
		and the canonical inclusion $\widehat{J}_0\subset W$.
		The output is a finite cover $\tilde{M}$ of $M$ together with a homomorphism:
			$$\phi:\,\pi_1(\tilde{M})\longrightarrow \pi_1(W),$$
		with described restrictions to its JSJ pieces.
		Since the CI completion $W$
		is an Eilenberg--MacLane space $K(\pi_1(W),1)$,
		it would be convenient to realize $\phi$ as a map
			$$f:\tilde{M}\longrightarrow W,$$
		which is unique up to homotopy.
		
		Suppose for the moment that we are provided with a representation:
			$$\eta_0:\pi_1(\widehat{J}_0)\longrightarrow\mathrm{Iso}_e\Sft,$$
		rather than a virtual representation.
		By Proposition \ref{propertiesCI} (2)	and Lemma \ref{class inversion},
		there is an extension over $\pi_1(W)$, (which is still denoted by $\eta_0$, regarding the original one
		as restriction,) so that
		the composition:
			$$\tilde{\rho}:\,\pi_1(\tilde{M})\stackrel{\phi}\longrightarrow \pi_1(W)\stackrel{\eta_0}\longrightarrow \mathrm{Iso}_e\Sft$$
		gives rise to a virtual extension of the representation
			$$\rho_0:\,\pi_1(J_0)\longrightarrow\pi_1(\widehat{J}_0)\stackrel{\eta_0|}\longrightarrow \mathrm{Iso}_e\Sft.$$
		
		At this basic level, the virtual extension is nothing but a finer version
		of Theorem \ref{virtualExtensionRep} for the special case of Seifert representations
		about mixed $3$-manifolds. It exhibits a factorization of $\tilde{\rho}$ through
		the CI completion $\pi_1(W)$.
		However, Proposition \ref{propertiesCI} (3) allows us to promote the above construction
		to deal with virtual representations of $\pi_1(\widetilde{J}_0)$.
		
		\subsubsection{Construction of $(\tilde{M}',\rho)$}
		Now suppose as of Theorem \ref{controlledvirtualSeifertRepresentation}
		that $\widehat{J}'_0$ is a finite cover of $\widehat J_0$,
		and
			$$\eta:\pi_1(\widehat{J}'_0)\longrightarrow\mathrm{Iso}_e\Sft$$
		is a Seifert representation of $\pi_1(\widehat{J}'_0)$. Denote by
			$$(W',\sigma_{W'})$$
		the CI completion of $\widehat{J}'_0$ with respect to all the elevations
		of $\gamma_1,\cdots,\gamma_s$. By Proposition \ref{propertiesCI} (3),
		there exists a finite covering map
			$$\kappa:W'\longrightarrow W$$
		which respects the free involutions and the distinguished inclusions.
		In particular,
		$\kappa$ extends the covering $\widehat{J}'_0\to\widehat{J}_0$
		preserving the degree.
		
		Remember that we have obtained a finite cover $\tilde{M}$
		and a map $f:\tilde{M}\to W$ for the basic level.
		Take any elevation of $f$ with respect to $\kappa$, denoted as:
			$$f':\tilde{M}'\longrightarrow W'.$$
		This means that the following diagram is commutative up to homotopy:
		$$\begin{CD}
			\tilde{M}'@> f'>> W'\\
			@VVV  @VV\kappa V\\
			\tilde{M}@>f>> W
		\end{CD}$$
		and $\tilde{M}'\to\tilde{M}$ is the covering of $\tilde{M}$
		which is minimal in the sense that it admits no intermediate covering of this property.
		(More concretely, one may replace $W$	with the mapping cylinder	$Y_f\simeq W$,
		and turn the map $f$ into an inclusion $\tilde{M}\to Y_f$,
		then any elevation $\tilde{M}'\to Y'_f$
		of $\tilde{M}$ in the corresponding finite cover $Y'_f\simeq W'$
		gives rise to some $f':\tilde{M}'\to Y'_f\to W'$ up to homotopy.)
		Since $W'$ is a finite cover of $W$,
		there are only finitely many such elevations $(\tilde{M}',f')$
		up to isomorphism between covering spaces	and homotopy.
		Moreover, the covering degree $[\tilde{M}':\tilde{M}]$ is bounded by $[W':W]$.
		Denote by
			$$\phi':\pi_1(\tilde{M}')\longrightarrow\pi_1(W')$$
		the homomorphism (up to conjugation) induced by $f'$.
		
		Provided with $\eta$ and $\phi'$ above, we extend $\eta$ to be
			$$\eta:\pi_1(W')\longrightarrow\mathrm{Iso}_e\Sft$$
		by Proposition \ref{propertiesCI} (1) and (3) and Lemma \ref{class inversion}.
		The finite cover
			$$\tilde{M}'\longrightarrow M$$
		and the representation:
			$$\rho:\,\pi_1(\tilde{M}')\stackrel{\phi'}\longrightarrow \pi_1(W')\stackrel{\eta}\longrightarrow \mathrm{Iso}_e\Sft$$
		are the claimed objects in the conclusion of Theorem \ref{controlledvirtualSeifertRepresentation}.
		
		Homomorphisms which have been presented can be summarized in the following commutative diagram:
		$$\begin{CD}
		@. \pi_1(\widehat{J}'_0) @> \eta| >> \mathrm{Iso}_e\Sft\\
		@. @VV\mathrm{incl}_\sharp V @VV \mathrm{Id} V\\
		\pi_1(\tilde{M}') @>\phi'>> \pi_1(W') @>\eta>> \mathrm{Iso}_e\Sft\\
		@VV\mathrm{cov}_\sharp V	@VV\kappa_\sharp V\\
		\pi_1(\tilde{M})@>\phi>>\pi_1(W)
		\end{CD}$$
		The homomorphisms $\phi$ and $\phi'$ are realized by maps $f$ and $f'$ respectively.
		The representation $\rho$ that we have constructed is the composition along the middle row.
		
		We are going to verify Theorem \ref{controlledvirtualSeifertRepresentation} (2) in the next three subsections.

		\subsubsection{Restriction to JSJ pieces}
		For any elevation $\tilde{J}'\subset\tilde{M}'$ of a JSJ piece $J\subset M$,
		$\tilde{J}'$ covers a JSJ piece $\tilde{J}$ of $\tilde{M}$.
		Since we have constructed $\phi$ using Theorem \ref{virtualExtensionRep},
		either the restriction of $\phi$ to $\pi_1(\tilde{J})$ has cyclic image,
		or $J$ is the distinguished hyperbolic piece $J_0$ and the restriction of
		$\phi$ to $\pi_1(\tilde{J})$ is one of the following compositions up to conjugation of $\pi_1(W)$:
			$$\pi_1(\tilde{J})\longrightarrow \pi_1(\widehat{J}_0)\longrightarrow \pi_1(W)$$
		or
			$$\pi_1(\tilde{J})\longrightarrow \pi_1(\widehat{J}_0)\longrightarrow \pi_1(W)\stackrel{\sigma_W}\longrightarrow \pi_1(W).$$		
		In the cyclic case, the restriction of $\phi'$ to $\pi_1(\tilde{J}')$ must also have cyclic image
		as $\kappa_\sharp$ is injective. Then the restriction of $\rho$ to $\pi_1(\tilde{J}')$ has cyclic image
		as well. In the other case, the first homomorphism of either composition factors through $\pi_1(J_0)$ via
		the Dehn filling, possibly after homotopy of $f$, we may assume that
		$\tilde{J}$ covers either $J_0$ or $\sigma_W(J_0)$ under the map $f$.
		As $f'$ is an elevation of $f$ with respect to $\kappa$,
		the elevation $\tilde{J}'$ of $\tilde{J}$
		covers either the unique elevation $J'_0$ of $J_0$ or
		the unique elevation $\sigma_{W'}(J'_0)$ of $\sigma_W(J_0)$
		in $W'$.  Note that $\eta$ is equivariant up to conjugacy with respect to
		the class inversions $\sigma_{W'}$ and
		$\nu$, (Proposition \ref{propertiesCI} and Lemma \ref{class inversion}).
		It follows that by taking
			$$\beta:\,\tilde{J}'\longrightarrow J'_0\longrightarrow \widehat{J}'_0$$
		the composition of the covering and the inclusion, the restriction of $\rho$ to $\pi_1(\tilde{J}')$
		is either $\beta^*(\eta)$ or $\beta^*(\nu\eta)$.
		This verifies Theorem \ref{controlledvirtualSeifertRepresentation} (1).
		
		\subsubsection{Count of degree}
		By the consideration about the restriction of $\rho$ to JSJ pieces of $\tilde{M}'$ above,
		we have seen that a JSJ piece $\tilde{J}'$ gives rise to the $\beta$-pull-back type restriction of $\rho$
		if and only if $\tilde{J}'$ covers a JSJ piece $\tilde{J}$ of $\tilde{M}$
		such that $\phi(\pi_1(\tilde{J}))$ is non-cyclic. The union of all such $\tilde{J}$ in $\tilde{M}$
		form a (disconnected)	finite cover $\tilde{\mathcal{J}}$ of the distinguished piece $J_0\subset M$,
		and the union of all $\beta$-pull-back type $\tilde{J}'$ in $\tilde{M}'$ is nothing but
		the preimage $\tilde{\mathcal{J}}'$ of $\tilde{\mathcal{J}}$ in $\tilde{M}'$.
		%
		Therefore, suppose $\alpha_0$
		is the ratio between the total degree of $\beta$-pull-back type JSJ pieces of $\tilde{M}'$ over $J_0$
		and the degree of $\tilde{M}'$:
		$$[\tilde{\mathcal{J}}':J_0]\,=\,\alpha_0\cdot[\tilde{M}':M],$$
		then we observe
		$$\alpha_0\,=\,\frac{[\tilde{\mathcal{J}}':J_0]}{[\tilde{M}':M]}
		\,=\,\frac{[\tilde{\mathcal{J}}':\tilde{\mathcal{J}}]\cdot[\tilde{\mathcal{J}}:J_0]}{[\tilde{M}':\tilde{M}]\cdot[\tilde{M}:M]}
		\,=\,\frac{[\tilde{\mathcal{J}}:J_0]}{[\tilde{M}:M]}.$$		
		Note that $\alpha_0$ depends only on $M$ and $J_0\to\widehat{J}_0$
		since $\tilde{M}$ and $\phi$ are constructed according to them,
		and $\alpha_0$
		is positive because $\tilde{\mathcal{J}}$ is non-empty by Theorem \ref{virtualExtensionRep}.

	\subsubsection{Count of volume} In a very similar situation as in the proof of Theorem
	\ref{virt-positive}, to compute the  volume of the representation
		$$\rho:\pi_1(\tilde{M}')\longrightarrow\mathrm{Iso}_e\Sft,$$
	it suffices to understand the contribution to the representation volume of $\rho$
	from the $\beta$-pull-back type JSJ pieces $\tilde{J}'$ of $\tilde{M}'$.
	Note that the map
		$$\beta:\,\tilde{J}'\stackrel{\mathrm{cov.}}\longrightarrow J'_0\stackrel{\mathrm{fill}}\longrightarrow \widehat{J}'_0$$
	factors through a unique hyperbolic Dehn filling $\tilde{K}'$ of $\tilde{J}'$ which covers
	$\widehat{J}'_0$ branching over elevations of the core curves $\gamma_i$ via a map $\widehat{\beta}$:
		$$\beta:\,\tilde{J}'\stackrel{\mathrm{fill}}\longrightarrow \tilde{K}'\stackrel{\widehat{\beta}}\longrightarrow \widehat{J}'_0$$
	The restriction of $\rho$ to $\pi_1(\tilde{J}')$ thus factorizes as:
		$$\pi_1(\tilde{J}')\stackrel{\mathrm{fill}_\sharp}\longrightarrow
		\pi_1(\tilde{K}')\stackrel{\widehat{\rho}}\longrightarrow \mathrm{Iso}_e\Sft,$$
	where $\widehat{\rho}$ equals the $\widehat{\beta}$-pull-back
	of $\eta$ or $\nu\eta$.
	Note that the class inversion $\nu$ of $\mathrm{Iso}_e\Sft$ is
	realized by the conjugation of an orientation-preserving isomorphism of $\Sft$, so
	$$\mathrm{Vol}_{\mathrm{Iso}_e\Sft}(\widehat{J}'_0;\,\eta)\,=\,\mathrm{Vol}_{\mathrm{Iso}_e\Sft}(\widehat{J}'_0;\,\nu\eta).$$
	It follows from the additivity principle (Theorem \ref{additive}) that  
	the contribution to the representation volume of $\rho$ from the piece $\tilde{J}'$ equals
	$\mathrm{Vol}_{\mathrm{Iso}_e\Sft}(\,\tilde{K}';\,\widehat{\rho}\,)$ and
	\begin{eqnarray*}
	\mathrm{Vol}_{\mathrm{Iso}_e\Sft}(\,\tilde{K}';\,\widehat{\rho}\,)&=&
	|\mathrm{deg}(\widehat{\beta})|\cdot\mathrm{Vol}_{\mathrm{Iso}_e\Sft}(\widehat{J}'_0;\,\eta)\\
	&=&\frac{[\tilde{J}':J_0]}{[\widehat{J}'_0:\widehat{J}_0]}\cdot\mathrm{Vol}_{\mathrm{Iso}_e\Sft}(\widehat{J}'_0;\,\eta).
	\end{eqnarray*}
	On the other hand,
	the contribution from any cyclic type JSJ piece $\tilde{J}'$ of $\tilde{M}'$ is always zero by Lemma \ref{Almost trivial representation}.
	Take the summation of contribution from all JSJ pieces, using the formula of $\alpha_0$
	in the degree count: 
	
	\begin{eqnarray*}
	\mathrm{Vol}_{\mathrm{Iso}_e\Sft}(\tilde{M};\,\rho)
	&=&\sum_{\tilde J'\in \tilde{\mathcal{J}}'}\mathrm{Vol}_{\mathrm{Iso}_e\Sft}(\tilde{K'};\,\widehat{\rho})\\
	&=&\sum_{\tilde J'\in \tilde{\mathcal{J}}'}\frac{[\tilde{J}':J_0]}{[\widehat{J}'_0:\widehat{J}_0]}\cdot\mathrm{Vol}_{\mathrm{Iso}_e\Sft}(\widehat{J}'_0;\,\eta)\\
	&=&\frac{[\tilde{\mathcal{J}}':J_0]}{[\widehat{J}'_0:\widehat{J}_0]}\cdot\mathrm{Vol}_{\mathrm{Iso}_e\Sft}(\widehat{J}'_0;\,\eta)\\
	&=&\alpha_0\cdot\frac{[\tilde{{M}}':M]}{[\widehat{J}'_0:\widehat{J}_0]}\cdot\mathrm{Vol}_{\mathrm{Iso}_e\Sft}(\widehat{J}'_0;\,\eta).
	\end{eqnarray*}
	or equivalently,
	 $$\frac{\mathrm{Vol}_{\mathrm{Iso}_e\Sft}(\tilde{M}';\,\rho)}{[\tilde{M}':M]}\,=\,\alpha_0\cdot\frac{\mathrm{Vol}_{\mathrm{Iso}_e\Sft}(\widehat{J}'_0;\,\eta)}{[\widehat{J}'_0:\widehat{J}_0]}.$$
	This completes the proof of Theorem \ref{controlledvirtualSeifertRepresentation} (2), therefore  the proof of Theorem
    \ref{controlledvirtualSeifertRepresentation}.
				
\section{On covering invariants}\label{Sec-onCoveringInvariants}
Although the covering property does not hold for the representation volumes \cite[Corollary 1.8]{DLW}, we can
stabilize them to obtain covering invariants in the following way.
%

\begin{definition}
	For any closed orientable 3-manifold $N$, define the \emph{covering Seifert volume} of $N$ to be
		$$\CSV(N)\,=\,\varprojlim_{\tilde{N}}\,\frac{\SV(\tilde{N})}{[\tilde{N}:N]},$$
	valued in $[0,+\infty]$,
	where $\tilde{N}$ runs over all the finite covers of $N$.
	Note that the limit exists because $\SV(\tilde{N})\,/\,[\tilde{N}:N]$
	is non-decreasing under passage to finite covers.
	Similarly one can define the \emph{covering hyperbolic volume} $\CHV(M)$.
\end{definition}

\begin{proposition}
If $\CSV$, or $\CHV$, is valued on $[0,+\infty)$ for a class $\mathcal{C}$ of closed orientable 3-manifolds, then it
satisfies both domination property and covering property for $\mathcal{C}$.
\end{proposition}

\begin{proof}
	We verify the statement for $\CSV$, the argument for $\CHV$ completely similar.
	
	To verify the domonation property, let $f:M\to N$ be any map of non-zero degree between $M,N\in\mathcal{C}$.
	By definition, for any $\epsilon>0$, there is a finite cover
	$\tilde N$ of $N$ such that
		$$\frac{\SV(\tilde N)}{[\tilde N: N]}\,>\,\CSV(N)-\epsilon.$$
	We have the following commutative diagram:
		$$\begin{CD}
			\tilde{M}@> \tilde f>> \tilde N\\
			@VVV  @VV\ V\\
			{M}@>f>> N
		\end{CD}$$
	for the pull-back cover $\tilde{M}$ of $M$ via $f$, which has degree at most $[\tilde{N}:N]$.
	Then we have  $[\tilde M:M]\cdot|\mathrm{deg}(f)|\,=\,[\tilde N:N]\cdot|\mathrm{deg}(\tilde f)|$,
	and $|\mathrm{deg}(f)|\geq |\mathrm{deg}(\tilde f)|$, and
	$\SV(\tilde M)\geq |\mathrm{deg}(\tilde f)|\cdot\SV(\tilde N)$.
	It follows that
		$$\frac{\SV(\tilde M)}{[\tilde M:M]}
		\,=\,\frac{\SV(\tilde M)\cdot|\mathrm{deg}(f)|}{[\tilde N:N]\cdot|\mathrm{deg}(\tilde{f})|}
		\geq\frac{|\mathrm{deg}(f)|\cdot\SV(\tilde N)}{[\tilde N: N]}\geq |\mathrm{deg}(f)|\cdot(\CSV(N)-\epsilon).$$
	Taking the limit over all $\tilde{M}$
	and $\epsilon\to 0+$, we have
	$$\CSV(M)\geq |\mathrm{deg}(f)|\cdot\CSV(N).$$
	
	To verify the covering property, suppose that $f:M\to N$ is a covering map,
	so $\mathrm{deg}(f)$ equals $[M:N]$.
	Then any finite cover $\tilde M$ of $M$ is also a finite cover of $N$.
	By definition we have
	$$\frac{\SV(\tilde M)}{[\tilde M:M]}\,=\,[M:N]\cdot\frac{\SV(\tilde M)}{[\tilde M:N]}
	\,\leq\, [M:N]\cdot\CSV(N)\,=\,|\mathrm{deg}(f)|\cdot\CSV(N).$$
	Taking the limit over all $\tilde{M}$, we have $\CSV(M)\leq|\mathrm{deg}(f)|\cdot\CSV(N)$.
	So indeed we have
	$$\CSV(M)\,=\,|\mathrm{deg}(f)|\cdot\CSV(N),$$
	where the other direction follows from the domination property.
	%
	%
\end{proof}
	
We post some further problems, updating those of \cite[Section 8]{DLW}.
%

\begin{problem}
	Does $\CSV(M)$ exist in $(0,+\infty)$	for every closed orientable non-geometric graph manifold $M$?
\end{problem}

A positive answer would provide a nowhere vanishing invariant with the covering property
in the class of closed orientable non-geometric graph manifolds.
Finding such an invariant was suggested by Thurston (\cite[Problem 3.16]{Ki}).
See \cite{LW,Ne,WW} for some attempts motivated by showing the uniqueness of covering degree
between graph manifolds.
The uniqueness is confirmed by \cite{YW} using combinatorial method and matrix theory.
%

\begin{problem}
		Determine the possible growth types and asymptotics of the virtual Seifert volume
		for closed orientable 3-manifolds with positive simplicial volume.
\end{problem}

We speak of the growth with respect to towers of finite covers, as the covering degree increases.
Theorem \ref{main-unbounded} shows that there are towers with super-linear growth.
The estimates of \cite{BG1} imply that the growth must be at most exponential.
	

\begin{problem}
	Is $\CHV(M)$ equal to the simplicial volume for every closed orientable 3-manifold $M$?
\end{problem}

This quantity is at most the simplicial volume (see Remark \ref{remark-unbounded})
and we suspect that the equality might be achieved.


\begin{thebibliography}{ZZZZ}

\bibitem[Ag]{Ag} I.~Agol,
\textit{The virtual Haken conjecture}, with an appendix by I.~Agol, D.~Groves, J.~Manning,
Documenta Math.~\textbf{18} (2013), 1045--1087.

\bibitem[BCG]{BCG} {G.~Besson, G.~Courtois, S.~Gallot},
{\it In\'egalit\'es de Milnor Wood g\'eom\'etriques}, Comment.~Math.~Helv.~\textbf{82} (2007),  753--803.

\bibitem[BG1]{BG1} {R.~Brooks and W.~Goldman}, {\it The Godbillon--Vey invariant of a transversely homogeneous foliation},
Trans.~Amer.~Math.~Soc.~\textbf{286} (1984), 651--664.

\bibitem[BG2]{BG2} \bysame, {\it Volumes in Seifert space},  Duke Math.~J.~\textbf{51} (1984), 529--545.

\bibitem[DLW]{DLW} P.~Derbez, Y.~Liu, S.~Wang,
\textit{Chern--Simons theory, surface separability, and volumes of 3-manifolds},
J.~Topol.~\textbf{8} (2015): 933--974.

\bibitem[DW1]{DW1} P.~Derbez and S.~Wang,
{\it Finiteness of mapping degrees and $\mathrm{PSL}(2,\mathbb{R})$-volume on graph manifolds},
Algebr.~Geom.~Topol.~\textbf{9} (2009), 1727--1749.


\bibitem[DW2]{DW2} \bysame, 
\textit{Graph manifolds have virtually positive Seifert voume},
J.~London Math.~Soc.~\textbf{86} (2012), 17--35.

\bibitem[DSW]{DSW} P.~Derbez, H.~Sun, S.~Wang, {\it Finiteness of mapping degree sets for 3-manifolds.}  Acta Math.~Sin.~(Engl.~Ser.) \textbf{27} (2011), no.~5, 807--812.

\bibitem[Ga]{Ga} A.~Gaifullin, \textit{Universal realisators for homology classes}, Geom.~Topol.~\textbf{17} (2013), 1745--1772.

\bibitem[Gr]{Gr} M.~Gromov, \textit{Volume and bounded cohomology}, Inst.~Hautes~\'Etudes~Sci.~Publ.~Math.\textbf{56}(1982), 5--99.

\bibitem[HW]{HW} F.~Haglund and D.~T.~Wise, \textit{Special cube complexes}, Geom.~Funct.~Anal.~\textbf{17} (2008),
1551--1620.

\bibitem[JS]{JS} {W.~Jaco and P.~B.~Shalen},  {\it Seifert fibered space in
 3-manifolds},  Mem.~Amer.~Math.~Soc.~\textbf{21} (1979).

\bibitem[Joh]{Joh}  {K.~Johannson}, \textit{Homotopy Equivalence of 3-Manifolds with
Boundary}, Lecture Notes in Math.~761, Springer--Verlag, Berlin, 1979.

\bibitem[KM]{KM} {J.~Kahn and V.~Markovic},  \textit{Immersing almost geodesic surfaces in a closed hyperbolic three manifold}.
Ann. of Math. (2) 175 (2012), no. 3, 1127--1190.

\bibitem[Ki]{Ki} {R.~Kirby}, \emph{Problems in low-dimensional topology},
Geometric Topology, Edited by H.~Kazez, AMS/IP, Vol.~2, International
Press, 1997.

\bibitem[Li]{Liu} {Y.~Liu}, {\it A characterization of virtually embedded subsurfaces in 3-manifolds},
Trans.~Amer.~Math.~Soc., to appear. Preprint available at \texttt{arXiv:1406.4674}, 28 pages.


\bibitem[LM]{LM} Y.~Liu and V.~Markovic, \textit{Homology of curves and surfaces in closed hyperbolic 3-manifolds},
Duke Math.~J.~\textbf{164} (2015), 2723--2808.

\bibitem[LW]{LW} {J.~Luecke and Y.~Wu}, {\it Relative Euler number and finite covers of graph manifolds},  Geometric Topology (Athens, GA, 1993),  pp.~80--103, Amer.~Math.~Soc., Providence, RI, 1997.

\bibitem[Ne]{Ne} W.~Neumann, {\it Commensurability and virtual fibration for graph manifolds}, Topol.~\textbf{36} (1997), no.~2, 355--378.


\bibitem[PW1]{PW1} P.~Przytycki and D.~T.~Wise, \textit{Graph manifolds with boundary are virtually special},
J. Topol. 7 (2014), no. 2, 419--435.

\bibitem[PW2]{PW2} \bysame, \textit{Mixed 3-manifolds are virtually special},  \texttt{arXiv:1205.6742}.

\bibitem[PW3]{PW3} \bysame, \textit{Separability of embedded surfaces in 3-manifolds}, Compos. Math. 150 (2014), no. 9, 1623--1630.

\bibitem[Re]{Re-rationality} {A.~Reznikov},
{\it Rationality of secondary classes},
J.~Diff.~Geom.~\textbf{43} (1996), no.~3, 674--692.

\bibitem[RW]{RW} J.~H.~Rubinstein and S.-C.~Wang, \textit{$\pi_1$-injective surfaces
in graph-manifolds}, Comment.~Math.~Helv.~\textbf{73} (1998),
499--515.

\bibitem[So]{So} T.~Soma, \textit{The Gromov invariant of links}, Invent.~Math.~\textbf{64} (1981), 445--454.

\bibitem[Su]{Sun} H.~Sun,  \textit{Virtual Domination of $3$-manifolds}, Geom.~Topol.~\textbf{19} (2015), 2277--2328.

\bibitem[Th1]{Th1} {W.~P.~Thurston},
{\it The Geometry and Topology of Three-Manifolds}, Lecture Notes, Princeton, 1977.

\bibitem[Th2]{Th2} \bysame, \emph{Three dimensional manifolds,
Kleinian groups and hyperbolic geometry}, Bull.~Amer.~Math.~Soc.~\textbf{316}, 1982, 357--381.

\bibitem[WW]{WW} {S.~Wang and Y.~Wu,} {\it Covering invariants and co-Hopficity
of 3-manifold groups}, Proc.~London Math.~Soc.~(3) \textbf{68} (1994), no.~1, 203--224.

\bibitem[Wi]{Wi} D.~T.~Wise, \textit{From Riches to Raags: 3-Manifolds, Right-Angled Artin Groups, and
Cubical Geometry}, CBMS Regional Conference Series in Mathematics, vol. 117, Washington, DC, 2012.
%
%

\bibitem[YW]{YW} F.~Yu and S.~Wang, {\it Covering degrees are determined by graph manifolds involved.} Comment.~Math.~Helv.~\textbf{74} (1999), 238--247.
\end{thebibliography}
\end{document}